\documentclass{article}
\usepackage{amsmath,amssymb,amsthm,enumerate}
\usepackage[all]{xy}
\usepackage{graphics}

\usepackage{hyperref}

\newtheorem{Lemma}{Lemma}[section]
\newtheorem{Proposition}[Lemma]{Proposition}

\theoremstyle{definition}

\newtheorem{Theorem}[Lemma]{Theorem}
\newtheorem{Definition}[Lemma]{Definition}
\newtheorem{Fact}[Lemma]{Fact}

\newtheorem{Remark}[Lemma]{Remark}
\newtheorem{Corollary}[Lemma]{Corollary}
\newtheorem{Question}{Question}
\newtheorem{Example}[Lemma]{Example}

\newtheorem{Figure}{Figure}

\makeatletter
\@addtoreset{equation}{section}
\makeatother

\title{A generalization of the Kobayashi-Oshima uniformly bounded multiplicity theorem
}

\author{Taito Tauchi\thanks{Institute 
of Mathematics for Industry,
Kyushu University,
Nishi-ku,
Fukuoka,
819-0395,
Japan,
E-mail adress: tauchi.taito.342@m.kyushu-ac.jp}
\footnote{This research is partially supported by 
Grant-inAid for JSPS Fellows (20J00101),
Japan Society for the Promotion of Science (JSPS).}}
\date{}

\begin{document}
\maketitle
\begin{abstract}
Let $P$ be a minimal parabolic subgroup of a
real reductive Lie group $G$ 
and $H$ a closed subgroup of $G$.
Then
it is proved by T. Kobayashi and T. Oshima
that
the regular representation $C^{\infty}(G/H)$
contains each irreducible representation of $G$ 
at most finitely many times
if 
the number of $H$-orbits on $G/P$ is finite.
Moreover,
they also proved that
the multiplicities are uniformly bounded 
if
the number of $H_{\mathbb C}$-orbits on $G_{\mathbb C}/B$
is finite,
where $G_{\mathbb C}, H_{\mathbb C}$ are complexifications of $G, H$, respectively,
and $B$ is a Borel subgroup of $G_{\mathbb C}$.
In this article,
we prove
that
the multiplicities of the representations of $G$ induced from a parabolic subgroup $Q$
in the regular representation on $G/H$ 
are uniformly bounded
if the number of $H_{\mathbb C}$-orbits on $G_{\mathbb C}/Q_{\mathbb C}$ is finite.
For the proof of this claim, we also show the uniform boundedness of the dimensions of the spaces of group invariant hyperfunctions
using the theory of holonomic ${\mathcal D}_{X}$-modules.
\end{abstract}
{\bf Keywords}: orbit decomposition, spherical variety, intertwining operator, multiplicity, holonomic system, hyperfunction.\\
{\bf MSC2020;} primary 22E46; secondary 22E45, 53C30.

\section{Introduction}
Let $G$ be a real reductive Lie group,
$H$ a closed subgroup,
and $Q$ a parabolic subgroup of $G$.
In this article,
we prove
that
the multiplicities of the representations of $G$ induced from $Q$
in the regular representation on $G/H$ 
are uniformly bounded
if the number of $H_{\mathbb C}$-orbits on $G_{\mathbb C}/Q_{\mathbb C}$ is finite,
where
$G_{\mathbb C}, H_{\mathbb C}$, and $Q_{\mathbb C}$ are complexifications of $G, H$, and $Q$, respectively.
For the proof of this claim, we also show the uniform boundedness of the dimensions of the spaces of group invariant hyperfunctions
using the theory of holonomic ${\mathcal D}_{X}$-modules.
We explain the motivation of this in the following subsections.
The main results are stated in 
Section \ref{Main-Theorems}.

\subsection{The Finite Multiplicity Theorem}
\label{FMT}
Let $G$ be a real reductive algebraic group 
and 
$H$ a real algebraic subgroup of $G$.
T. Kobayashi and T. Oshima established 
a finiteness criterion for multiplicities of the regular representation on 
the homogeneous space $G/H$.
\begin{Fact}[{\cite[Thm.\:A]{KO}}]
\label{KO}
Suppose that $G$ and $H$ are defined algebraically over ${\mathbb R}$.
Then 
the following two conditions on the pair ($G, H$) are equivalent:
\begin{enumerate}[(i)]
\item
$\dim {\rm Hom}_{G}
(\pi,C^{\infty}(G/H,\tau))<\infty$
for any $(\pi,\tau)\in \hat{G}_{{\rm smooth}}\times\hat{H}_{{\rm f}}$,
\item
$G/H$ is real spherical.
\end{enumerate}
\end{Fact}

\begin{Remark}
\label{CV1}
	In \cite{KO}, an explicit upper bound of the dimensions in (i) of Fact \ref{FMT}
	was also given.
\end{Remark}

Here,
$\hat{G}_{{\rm smooth}}$
denotes the set of equivalence classes of 
irreducible smooth admissible Fr\'echet representations of $G$ with moderate growth,
and $\hat{H}_{{\rm f}}$ that of irreducible finite-dimensional representations of $H$. 
Given $\tau\in\hat{H}_{{\rm f}}$, 
we write $C^{\infty}(G/H,\tau)$ for the Fr\'echet space of smooth sections of the $G\mathchar`-$homogeneous vector bundle over $G/H$ associated to $\tau$.
The terminology {\it real sphericity} was introduced by Kobayashi \cite{RS} in his study of a broader framework for global analysis on homogeneous spaces than the usual (e.g., reductive symmetric spaces).

\begin{Definition}
A homogeneous space $G/H$ is {\it real spherical} if a minimal parabolic subgroup $P$ of $G$ has an open orbit on $G/H$.
\end{Definition}

The following is one of
the characterizations of real spherical homogeneous spaces.
This is a consequence of
the rank one reduction of T. Matsuki \cite{Matsuki}
and
the classification of real spherical homogeneous spaces of real rank one by B. Kimelfeld 
\cite{Kimelfeld}.
\begin{Fact}[{\cite{Bien}}]
\label{KM}
For the pair
$(G,H)$,
the following two conditions are equivalent:
\begin{enumerate}[(i)]
\setcounter{enumi}{1} 
\item
$G/H$ is real spherical, (i.e., there exists an open $P$-orbit on $G/H$),
\item
$\#(H\backslash G/P)<\infty$.
\end{enumerate}
\end{Fact}
Therefore, for a minimal parabolic subgroup $P$, the three conditions (i), (ii), and (iii) are equivalent by Facts \ref{KO} and \ref{KM} (see Figure \ref{fig1} below).  
Thus 
we ask a question what will happen to the relationship among the three conditions, if we replace $P$ by a general parabolic subgroup $Q$ of $G$.
There is an obvious extension of the
conditions (ii) and (iii)
to a general parabolic subgroup $Q$
(see Definition \ref{DefQ} below).
In order to formulate a variant of (i) for a parabolic subgroup $Q$ of $G$,
we review the notion of $Q$-series.
\begin{Definition}[{\cite[Def.\:6.6]{Kobshintani}}]
\label{DefQ}
Let $\pi \in \hat{G}_{\rm smooth}$.
We say that $\pi$
belongs to $Q${\it -series} if $\pi$ occurs as a subquotient of the degenerate principal series representation $C^{\infty}(G/Q,\tau)$ for some $\tau \in \hat {Q}_{\rm f}$.
\end{Definition}

\vspace{0.2cm}
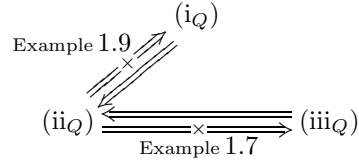
\begin{minipage}{0.47\hsize}
\begin{Figure}
\label{fig1}
$
P:\text{minimal parabolic}$
\begin{center}
$
\xymatrix{
& (\rm{i}) \ar@{<=>}[ld]_{{\rm Fact} \ref{KO}} \ar@{<=>}[rd] & \\
(\rm{ii}) \ar@{<=>}[rr]_{{\rm Fact} \ref{KM}}& & (\rm{iii})
}
$
\end{center}
\end{Figure}
\end{minipage}
\begin{minipage}{0.47\hsize}
\begin{Figure}
\label{fig2}
$Q:\text{general parabolic}$
\begin{center}
$
\xymatrix{
& ({\rm i}_{Q}) \ar@<0.1cm>@{=>}[ld] & \\
({\rm ii}_{Q}) \ar@<0.1cm>@{=>}[ru]|{\times}^{{\rm Example}\:\ref{examinfdim}}
\ar@<0.1cm>@{<=}[rr]& & ({\rm iii}_{Q}) \ar@<0.1cm>@{<=}[ll]|{\times}^{{\rm Example}\:\ref{HGQinfty}}
}
$
\end{center}
\end{Figure}
\end{minipage}
\vspace{0.2cm}

For a parabolic subgroup $Q$ of $G$,
we set
$$\hat{G}_{{\rm smooth}}^{Q}:=\{\pi\in\hat{G}_{\rm smooth}\mid \pi\:\text{belongs to $Q$-series}\}.$$
Obviously, $\hat{G}_{{\rm smooth}}^{Q}\supset \hat{G}_{{\rm smooth}}^{Q'}$ if $Q\subset Q'$.
Moreover,
$\hat{G}^{Q}_{\rm smooth}$ is equal to $\hat{G}_{\rm smooth}$ if $Q=P$ (minimal parabolic) by Harish-Chandra's subquotient theorem \cite{subquo} and to $\hat{G}_{{\rm f}}$ if $Q=G$.
\begin{Definition}
\label{DefQ}
For a parabolic subgroup $Q$ of $G$, we define three conditions (i$_{Q}$), (ii$_{Q}$), and (iii$_{Q}$) as follows:
\begin{enumerate}[(i$_{Q}$)]
\item
$\dim {\rm Hom}_{G}
(\pi,C^{\infty}(G/H,\tau))<\infty$
for all $(\pi,\tau)\in \hat{G}^{Q}_{{\rm smooth}}\times\hat{H}_{{\rm alg}}$,
\item
$Q$ has an open orbit on $G/H$,
\item
$\#(H\backslash G/Q)<\infty$.
\end{enumerate}
\end{Definition}
The conditions (i$_{Q}$), (ii$_{Q}$), and (iii$_{Q}$) reduce to (i), (ii), and (iii), respectively, if $Q=P$ (minimal parabolic), and we have seen in Facts \ref{KO} and \ref{KM} that the following equivalences hold for $Q=P$,
\begin{align*}
({\rm i}_{P}) \iff ({\rm ii}_{P}) \iff ({\rm iii}_{P}).\\
\intertext{Furthermore,
if $Q=G$,
the condition (i$_{Q}$) automatically holds by the Frobenius reciprocity, while
(ii$_{Q}$) and (iii$_{Q}$) are obvious.  Hence}
({\rm i}_{G}) \iff ({\rm ii}_{G}) \iff ({\rm iii}_{G}).
\end{align*}
For a general parabolic subgroup $Q$, clearly, (iii$_{Q}$) implies (ii$_{Q}$).
However 
there is an easy counterexample for 
the converse statement.
\begin{Example}
\label{HGQinfty}
The real projective space ${\mathbb R}{\mathbb P}^{2}=SL(3,{\mathbb R})/Q=G/Q$
splits into an open orbit and continuously many fixed points of the
unipotent radical $H$ of the opposite parabolic subgroup $\overline{Q}$ of $Q$.
\end{Example}

On the other hand, the implication (i$_{Q}$) $\Rightarrow$ (ii$_{Q}$) holds. 
To see this,
we define a subset $\hat{H}_{\rm f}(G)$ of $\hat{H}_{\rm f}$ by
$$
\hat{H}_{\rm f}(G):=
\{
\tau\in \hat{H}_{\rm f}\mid
\tau 
\text{ appears as a quotient of some element of }
\hat{G}_{\rm f} 
\}.
$$
The implication (i$_{Q}$) $\Rightarrow$ (ii$_{Q}$) is derived 
from the following stronger assertion:
\begin{Fact}[{\cite[Cor.\:6.8]{Kobshintani}}]
\label{K}
If
$Q$ does not have an open orbit on $G/H$,
then,
for any $\tau \in \hat{H}_{\rm f}(G)$, there exists $\pi \in \hat{G}^{Q}_{\rm smooth}$ such that $\dim{\rm Hom}_{G}(\pi, C^{\infty}(G/H,\tau))=\infty$.
\end{Fact}

This fact implies that 
(i$_{Q}$)$\Rightarrow$(ii$_{Q}$) is true.
However,
(ii$_{Q}$)$\Rightarrow$(i$_{Q}$)
is not always true for a general parabolic subgroup $Q$.

\begin{Example}
\label{examinfdim}
$\dim {\rm Hom}_{G} (C^{\infty}(G/Q),C^{\infty}(G/H))=\infty$ 
for the pair $(G,H,Q)$ in Example \ref{HGQinfty}.
\end{Example}

We summarize the known relationship among the three conditions in Figure \ref{fig2},
which shows that
the relation between the conditions (i$_{Q}$) and (iii$_{Q}$) is unsettled.

\begin{Question}
Determine the relationship between the following two conditions:
\begin{enumerate}[(i$_{Q}$)]
\item
$\dim {\rm Hom}_{G}
(\pi,C^{\infty}(G/H,\tau))<\infty$
for all $(\pi,\tau)\in \hat{G}^{Q}_{{\rm smooth}}\times\hat{H}_{{\rm alg}}$,

\setcounter{enumi}{2}
\item
$\#(H\backslash G/Q)<\infty$.
\end{enumerate}
\end{Question}
For this question,
we proved 
that
there exists the pair 
$(G,H,Q)$
satisfying
the condition (iii$_{Q}$),
although it does NOT satisfy the condition 
(i$_{Q}$)
in \cite[Thm.\:1.8]{Tauchi-QH},
and
proved
that
(i$_{Q}$)
$\Rightarrow$
(iii$_{Q}$)
holds
under a certain condition of orientation 
in
\cite[Thm.\:2.4]{Tauchi-PJA}.
Figure \ref{fig3} given below
summarises
the relationship among
the conditions
(i$_{Q}$),
(ii$_{Q}$)
and
(iii$_{Q}$).
In this figure,
$\Delta$ on the arrow of
(i$_{Q}$)
$\Rightarrow$
(iii$_{Q}$)
means that
this is
proved only under the additional assumption.
\begin{Figure} 
\label{fig3}
\begin{center}
$Q:\text{general parabolic}$
\end{center}
\samepage
\begin{center}
$
\xymatrix{
& ({\rm i}_{Q}) \ar@<0.1cm>@{=>}[ld] 
\ar@<0.1cm>@{<=}[rd]|{\times}^{\cite{Tauchi-QH}} 
\ar@<-0.1cm>@{=>}|{\Delta}[rd]_{\cite{Tauchi-PJA}}& \\
({\rm ii}_{Q}) \ar@<0.1cm>@{=>}[ru]|{\times}^{\text{Example} \ref{examinfdim}}
\ar@<0.1cm>@{<=}[rr]& & ({\rm iii}_{Q}) \ar@<0.1cm>@{<=}[ll]|{\times}^{\text{Example} \ref{HGQinfty}}
}
$

\end{center}
\end{Figure}

\subsection{Uniformly Bounded Multiplicity}
By Fact \ref{KO},
the finiteness of the number of $P$-orbits on $G/H$
guarantees
that of multiplicities in the regular representation on $G/H$.
Kobayashi and
Oshima
also proved 
the criterion for uniformly boundedness 
of multiplicities
in the regular representation.
In this article,
we say that 
a complex Lie group
$L_{\mathbb C}$
is a complexification of a Lie group $L$
if
$L_{\mathbb C}$ contains $L$ as a closed subgroup
and
${\mathfrak l}_{\mathbb C}=
{\mathfrak l}\oplus \sqrt{-1}{\mathfrak l}$,
where ${\mathfrak l}_{\mathbb C}$
is the Lie algebra of $L_{\mathbb C}$.

\begin{Fact}[{\cite[Theorem B]{KO}}]
\label{KOB}
The following two conditions on the pair
$(G,H)$
are equivalent:
\begin{enumerate}[(I)]
\item
$\displaystyle
\sup_{\tau \in \hat{H}_{\rm f}}
\sup_{\pi\in\hat{G}_{{\rm smooth}}}
\frac{1}{\dim \tau}
\dim {\rm Hom}_{G}
(\pi,C^{\infty}(G/H,\tau))<\infty
$,
\item
$G_{\mathbb C}/H_{\mathbb C}$
is spherical.
\end{enumerate}
\end{Fact}
\begin{Remark}
\label{CV2}
	In \cite{KO},
	an explicit upper bound of (I)
	of Fact \ref{KOB} was also given,
	which is optimal in the case that $H$ is 
	the maximal unipotent subgroup $N$ of $G$.
\end{Remark}

Here we say that
a homogeneous space
$G_{\mathbb C}/H_{\mathbb C}$
is spherical 
if
a Borel subgroup of $G_{\mathbb C}$
has an open orbit on
$G_{\mathbb C}/H_{\mathbb C}$.
It is well known
that
the condition (II)
in Fact \ref{KOB}
is characterized by the finiteness of the number of $B$-orbits on $G_{\mathbb C}/H_{\mathbb C}$. 
\begin{Fact}[\cite{Brion,Matsuki,Vinberg}]
\label{BMV}
The condition (II) in
Fact \ref{KOB}
is equivalent to the following condition:
\begin{enumerate}[(I)]
\setcounter{enumi}{2}
\item
$\#(B\backslash G_{\mathbb C}/H_{\mathbb C})<\infty$.
\end{enumerate}
\end{Fact}
Therefore,
for a Borel subgroup $B$,
the three conditions
(I),
(II),
and
(II)
are equivalent by
Facts
\ref{KOB}
and
\ref{BMV}
just like the case of Figure \ref{fig1}
(see Figure \ref{fig4} given below).
This equivalence can be interpreted that
the $B$-orbit decomposition of $G_{\mathbb C}/H_{\mathbb C}$
has some information about 
uniformly boundedness of the 
multiplicities in the regular representation on $G/H$.
\begin{Figure}
\label{fig4}
\begin{center}
$B:\text{Borel subgroup}$
\end{center}
\samepage
\begin{center}
$
\xymatrix{
& (\rm{I}) \ar@{<=>}[ld]_{{\rm Fact} \ref{KOB}} \ar@{<=>}[rd] & \\
(\rm{II}) \ar@{<=>}[rr]_{{\rm Fact} \ref{BMV}}& & (\rm{III})
}
$
\end{center}
\end{Figure}

On the other hand,
in Section \ref{FMT},
we consider the relationship between
the $Q$-orbit decomposition of $G/H$
and
the multiplicities of $Q$-series representations
in the regular representation on $G/H$
motivated by the equivalence in Figure \ref{fig1}.
Thus,
Figure \ref{fig4}
makes us think about the following question:
\begin{Question}
\label{Q2}
Determine the relationship between
the $Q_{\mathbb C}$-orbit decomposition
of $G_{\mathbb C}/H_{\mathbb C}$
and
the uniform boundedness of
$Q$-series representations
in regular representation on $G/H$.
\end{Question}

\subsection{Main Theorems}
\label{Main-Theorems}
In this article,
we prove a variant of Fact \ref{KOB}
for a (not necessary minimal)
parabolic subgroup $Q$
as an answer of Question \ref{Q2}.
\begin{Theorem}
\label{Qfin}
Let
$Q$ be a parabolic subgroup of a real reductive Lie group $G$,
and
$H$ a closed subgroup of $G$.
We write $G_{\mathbb C}$,
$H_{\mathbb C}$
and
$Q_{\mathbb C}$
for complexifications
of 
$G,H$ and $Q$,
respectively.
If
the number of connected components of $H_{\mathbb C}$ is finite and
$\#(H_{\mathbb C}\backslash G_{\mathbb C}/Q_{\mathbb C})<\infty$,
then we have
\begin{equation}
\sup_{(\eta,\tau)\in\hat{Q}_{\rm f}\times\hat{H}_{\rm f}}
\frac{1}{\dim \eta\cdot \dim\tau}
\dim{\rm Hom}_{G}(C^{\infty}(G/Q,\eta),
C^{\infty}(G/H,\tau))<\infty.
\label{/dimdim}
\end{equation}
\end{Theorem}

For the proof of Theorem
\ref{Qfin},
we give an upper bound of the dimensions
of 
the relative invariant Sato hyperfunction spaces 
with respect to a group action
by using the theory of holonomic ${\mathcal D}_{X}$-modules.
Let
${\mathcal B}_{M}$
be the sheaf of Sato's hyperfunctions on
a real analytic manifold $M$.
We say that a complex manifold $X$ 
is a complexification of $M$
if $X$ contains $M$ as a real analytic submanifold and $T_{x}X= T_{x}M\oplus \sqrt{-1}T_{x}M$
for any $x\in M$.
\begin{Theorem}
\label{Solfin-analytic}
Let 
$X$
be a complexification of a real analytic manifold $M$.
Suppose that 
a complex Lie group
$H_{\mathbb C}$ 
with finitely many connected components
acts on 
$X$
with
$\#(H_{\mathbb C}\backslash X)<\infty$.
Then,
for any relatively compact semianalytic open subset $U\subset M$,
there exists
$C>0$
such that
for any finite-dimensional representation $\tau$
of
${\mathfrak h}_{{\mathbb C}}$,
we have
\begin{align}
\dim
(\Gamma(U;{\mathcal B}_{M})\otimes \tau)^{{\mathfrak h}_{\mathbb C}}
\leq C\cdot \dim\tau.
\label{GammaBMain}
\end{align}
\end{Theorem}

Moreover,
we can give an alternative approach of (II)$\Rightarrow$(I) in Fact \ref{KOB}
as a corollary of
Theorem \ref{Solfin-analytic}.

\begin{Remark}
As stated in Remarks \ref{CV1} and \ref{CV2},
explicit upper bounds of multiplicities were already given in \cite{KO}.
Using the method of this article,
one can show that the left-hand side of \eqref{/dimdim}
is bounded by
\begin{eqnarray*}
	\sum_{p=0}^{\dim G_{\mathbb C}/Q_{\mathbb C}}
	c_{p}\frac{(\dim G_{\mathbb C}/Q_{\mathbb C})!}
	{(\dim G_{\mathbb C}/Q_{\mathbb C}-p)!}\cdot (2k_{0})^{p},
\end{eqnarray*}
which is not optimal 
because this is the sum of upper bounds of the dimensions 
of invariant hyperfunctions supported by each $H$-orbit on $G/Q$
(Intertwining operators can be regarded as
invariant 
distributions,
see Fact \ref{GPGP}).
Here,
$k_{0}$ is the maximum value of heights of the roots of ${\mathfrak g}_{\mathbb C}$ 
relative to some Cartan subalgebra
and $c_{p}$ is the sum 
of the numbers of connected components
of ${\mathcal O}\cap G/Q$
for all $p$-codimensional 
$H_{\mathbb C}$-orbits ${\mathcal O}$ on $G_{\mathbb C}/Q_{\mathbb C}$.
\end{Remark}

\begin{Remark}
In addition to the original proof of the implication (II) $\Rightarrow$ (I) in Fact \ref{KOB} given in \cite{KO},
Kobayashi suggested an alternative approach by using the theory of holonomic ${\mathcal D}$-modules in Bonn, September 2011.
In this direction,
there is a recent work by 
A. Aizenbud,
D. Gourevitch and A. Minchenko \cite{AGM}
by an approach of holonomic ${\mathcal D}$-modules,
however,
there are also some differences between 
Theorem \ref{Solfin-analytic}
and their results.
Their proof utilizes an argument of the Weil representation
and a different filtration from ours,
and their estimate of the multiplicity is not strong enough to
recover the implication (II) $\Rightarrow$ (I) in Fact \ref{KOB}.
We note that 
\cite[Thm.\:D]{AGM}
for bundle-valued tempered distributions in the algebraic setting
can be deduced from Theorem \ref{Solfin-analytic}
for hyperfunctions in the analytic setting.
	\end{Remark}

This article is organized as follows.
In Section
\ref{preliminary},
we recall the theory of sheaves and
${\mathcal D}_{X}$-modules.
In
Section
\ref{Holonomy-dimension},
we
give an upper bound of 
the dimensions of
the Sato hyperfunction solution spaces 
for holonomic ${\mathcal D}_{X}$-modules.
We prove Theorem \ref{Solfin-analytic}
in Section
\ref{proofSolfin}.
Theorem \ref{Qfin}
is proved in
Section
\ref{proofQfin}.
An alternative approach of
(II)$\Rightarrow$(I) in Fact \ref{KOB} is given in Section \ref{APB}.

\section*{Acknowledgment}
The author is grateful to Professor Toshiyuki Kobayashi for his helpful advice and constant encouragement.
He also thanks the anonymous referee for his/her careful reading of  the manuscript and many helpful comments and suggestions. 

\section{Preliminaries of Sheaves and ${\mathcal D}_{X}$-modules}
\label{preliminary}
In this section,
we recall the basic notions of the theory of 
sheaves and 
${\mathcal D}_{X}$-modules.
Although almost all the materials in this section
are well known,
we prove some results for the convenience.

\subsection{Sheaves}
In this subsection,
we briefly recall the basic notions of the theory
of
sheaves.
For 
further references on this subject, see
\cite{Kas},
for example.

Let $X$ be a good topological space
(i.e., a
Hausdorff, locally compact space which is countable at infinity and has finite flabby dimension),
and ${\mathbb C}_{X}$ the constant sheaf on $X$.
We write
${\rm Mod}({\mathbb C}_{X})$ for
the abelian category of sheaves of
${\mathbb C}$-vector spaces on $X$
and
${\rm D}^{b}({\mathbb C}_{X})$
for the bounded derived category of
${\rm Mod}({\mathbb C}_{X})$.
Identifying  
a sheaf ${\mathcal S}\in{\rm Mod}({\mathbb C}_{X})$ 
with 
a complex of sheaves
$$
\dots\to0\to{\mathcal S}\to0\to\dots\qquad \in{\rm D}^{b}({\mathbb C}_{X})
$$ (with ${\mathcal S}$ in degree $0$),
we regard ${\rm Mod}({\mathbb C}_{X})$
as a full subcategory of ${\rm D}^{b}({\mathbb C}_{X})$.

For a closed subset $\iota_{K}\colon K\hookrightarrow X$
and
${\mathcal S}\in{\rm Mod}({\mathbb C}_{X})$,
we define ${\mathcal S}_{K}\in{\rm Mod}({\mathbb C}_{X})$ by
\begin{eqnarray*}
{\mathcal S}_{K}:=\iota_{K*}\iota_{K}^{-1}({\mathcal S}),
\end{eqnarray*}
where
$\iota_{K*}$
and
$\iota_{K}^{-1}$
are the direct image functor
and the inverse image functor,
respectively.
Then, because
$(\iota_{K}^{-1},\iota_{K*})$
is an adjoint pair of functors,
we have
\begin{equation}
{\rm Hom}_{{\rm Mod}({\mathbb C}_{K})}(\iota_{K}^{-1}({\mathcal S}),\iota_{K}^{-1}({\mathcal S}))
\simeq {\rm Hom}_{{\rm Mod}({\mathbb C}_{X})}({\mathcal S},\iota_{K*}\iota_{K}^{-1}({\mathcal S})).
\label{adjunction}
\end{equation}
Thus, there exists a natural map
${\mathcal S}\to{\mathcal S}_{K}$,
which corresponds to the identity of the left-hand side of \eqref{adjunction}.
Then,
for an open subset $U\subset X$,
we define
${\mathcal S}_{U}:=\operatorname{Ker}({\mathcal S}\to{\mathcal S}_{K})$,
where $K:=X\backslash U$.
For a locally closed subset $Z\subset X$,
we take an open subset $U\subset X$
and a closed subset $K\subset X$
satisfying $Z=U\cap K$
and define
${\mathcal S}_{Z}:=({\mathcal S}_{U})_{K}$.
This definition does not depend on the choice of $U$ and $K$.
We sometimes abbreviate 
$({\mathbb C}_{X})_{Z}$
to
${\mathbb C}_{Z}$.
We write ${\rm Mod}_{K}({\mathbb C}_{X})$
for the full subcategory of 
${\rm Mod}({\mathbb C}_{X})$
consisting objects whose supports are contained in 
a closed subset $K\subset X$.
\begin{Lemma}
\label{iiid}
For ${\mathcal S}\in{\rm Mod}_{K}({\mathbb C}_{X})$,
we have
${\mathcal S}_{K}\simeq {\mathcal S}$.
\end{Lemma}
\begin{proof}
By the definition,
we have an exact sequence
\begin{eqnarray}
0\to
{\mathcal S}_{X\backslash K}\to
{\mathcal S}\to
{\mathcal S}_{K}\to
0.
\label{SSS}
\end{eqnarray}
Taking stalks at each point,
we have $S_{X\backslash K}=0$.
Thus,
\eqref{SSS} implies the lemma.
\end{proof}

\begin{Lemma}
\label{gammaKS}
For a closed subset $\iota_{K}\colon K\hookrightarrow X$
and
${\mathcal S}\in{\rm Mod}_{K}({\mathbb C}_{X})$,
we have $\Gamma(K;\iota_{K}^{-1}{\mathcal S})=\Gamma(X;S)$,
where
$\Gamma(X;{\mathcal S})$
is the space of sections of ${\mathcal S}$ on  $X$.
\end{Lemma}
\begin{proof}
	${\mathcal S}\simeq {\mathcal S}_{K}\simeq \iota_{K*}\iota_{K}^{-1}({\mathcal S})$
	implies
	$
	\Gamma(X;S) = \Gamma(X;\iota_{K*}\iota_{K}^{-1}({\mathcal S}))
	=
	\Gamma(K;\iota_{K}^{-1}{\mathcal S}).
	$
\end{proof}

For a closed subset $K$ of an open subset $U\subset X$,
we define a subspace of 
$\Gamma(U;{\mathcal S})$
by
$$
\Gamma_{K}(U;{\mathcal S}):=\operatorname{Ker}(\Gamma(U;{\mathcal S})
\to\Gamma(U\backslash K;{\mathcal S})).
$$
For a locally closed subset $Z\subset X$,
we take an open subset $U\subset X$
and a closed subset $K\subset X$
satisfying $Z=U\cap K$
and set
$
\Gamma_{Z}(U;{\mathcal S}):=\Gamma_{U\cap Z}(U;{\mathcal S})
$.
Then,
for a locally closed subset $Z\subset X$, 
the sheaf
$\Gamma_{Z}({\mathcal S})\in {\rm Mod}({\mathbb C}_{X})$
is defined
by
\begin{align*}
\Gamma(U;\Gamma_{Z}({\mathcal S})):=
\Gamma_{Z\cap U}(U;{\mathcal S})
\end{align*}
for an open subset $U\subset X$.
Then,
$\Gamma_{Z}$
is an endofunctor on
${\rm Mod}({\mathbb C}_{X})$.
We write ${\mathbb R}\Gamma_{K}\colon {\rm D}^{b}({\mathbb C}_{X})\to{\rm D}^{b}({\mathbb C}_{X})$
for the right derived functor of $\Gamma_{Z}$
and
${\mathbb R}^{k}\Gamma_{Z}$
for the $k$-th cohomology.
We note that $\Gamma_{U}=j_{*}\circ j^{-1}$
for an open subset
$j\colon U\hookrightarrow X$
and that ${\rm supp}(\Gamma_{K}({\mathcal S}))\subset K$ for a closed subset $K\subset X$.

We quote the following two properties of 
the functor $\Gamma_{Z}$.
See, 
\cite[Prop.\:2.3.9]{Kas}
or
\cite[Prop\:4.14 in Appx.\:II]{Bjork},
for example.
\begin{Lemma}
\label{ZZZ}
Let $Z$ be a locally closed subset of $X$
and
$Z'$ a closed subset of
$Z$.
Then,
for
${\mathcal S}\in {\rm Mod}({\mathbb C}_{X})$,
we have an exact sequence
\begin{eqnarray*}
0\to
\Gamma_{Z'}({\mathcal S})\to
\Gamma_{Z}({\mathcal S})\to
\Gamma_{Z\backslash Z'}({\mathcal S}).
\end{eqnarray*}
\end{Lemma}
\begin{Lemma}
\label{KK'}
Let
$K$ and $K'$
be closed subsets of $X$.
Then we have
\begin{eqnarray*}
{\mathbb R}\Gamma_{K'}\circ
{\mathbb R}\Gamma_{K}(\cdot)\simeq
{\mathbb R}\Gamma_{K\cap K'}(\cdot).
\end{eqnarray*}
\end{Lemma}

For a morphism 
$f\colon X\to Y$
of good topological spaces,
we write
$f_{!}\colon {\rm Mod}({\mathbb C}_{X})\to {\rm Mod}({\mathbb C}_{Y})$
for the proper direct image functor.
Namely,
for
${\mathcal S}\in {\rm Mod}({\mathbb C}_{X})$,
we have
$$
\Gamma(U;f_{!}{\mathcal S})=
\{
s\in\Gamma(f^{-1}(U);{\mathcal S})\mid
f\colon\operatorname{supp}(s)\to U\:\text{is proper}\},
$$
where
$U$ is an 
open subset of $Y$.
We also write
$f^{!}$
for the right adjoint functor of
the right derived functor
${\mathbb R}f_{!}$
of $f_{!}$.

Let 
$\{pt\}$
be a topological space with one element.
Then, 
there exists
a natural map 
$a_{X}\colon X\to\{pt\}$.
The dualizing complex $\omega_{X}$ on $X$ is defined by
$\omega_{X}:=a_{X}^{!}({\mathbb C}_{\{pt\}})\in {\rm D}^{b}({\mathbb C}_{X})$.
If $X$ is a real manifold,
the orientation sheaf
$or_{X}$ is defined
by
\begin{eqnarray}
or_{X}:=\omega_{X}[-\dim X]\qquad \in {\rm Mod}({\mathbb C}_{X}),
\label{or}
\end{eqnarray}
where
$[-\dim X]$ is the shift functor.
Note that 
$or_{X}$ is locally isomorphic to ${\mathbb C}_{X}$.
Namely,
for any point $x\in X$,
there exists an open neighborhood $U\subset X$,
we have $or_{X}|_{U}\simeq {\mathbb C}_{X}|_{U}\in{\rm Mod}({\mathbb C}_{U})$.
Here, we set ${\mathbb C}_{X}|_{U}:=\iota_{U}^{-1}({\mathbb C}_{X})$
for an open embedding $\iota_{U}\colon U\hookrightarrow X$.

At the end of this section,
we prove Lemma \ref{MN},
which will be used in later sections.
For this,
we need Lemma \ref{ff}
whose proof can be found in 
\cite[Prop.\:3.1.12]{Kas}
for example.
\begin{Lemma}
\label{ff}
Let
$f\colon X\to Y$
be a morphism of good topological spaces.
Assume
$X$ is diffeomorphic to some locally closed subset of $Y$
by $f$.
Then we have
$$
f^{!}(\cdot)\simeq
f^{-1}\circ {\mathbb R}\Gamma_{f(X)}(\cdot).
$$
In particular,
if $f\colon X\to Y$ is a closed embedding,
Lemma \ref{iiid} implies
$$
f_{*}\circ f^{!}(\cdot)\simeq
 {\mathbb R}\Gamma_{f(X)}(\cdot)
$$
by exactness of $f_{*}\circ f^{-1}(\cdot)=(\cdot)_{f(X)}$.
\end{Lemma}
\begin{Lemma}
\label{MN}
Let 
$L$
be a real analytic manifold,
$M$
a closed $m$-dimensional submanifold of $L$,
and
$N$
a closed $n$-dimensional submanifold of $M$.
Then,
$
{\mathbb R}\Gamma_{N}({\mathbb C}_{M})
$
is locally isomorphic to
$
{\mathbb C}_{N}[-(m-n)]
$
as an object of
${\rm D}^{b}({\mathbb C}_{L})$.
\end{Lemma}
\begin{proof}
By \eqref{or},
we have locally
${\mathbb C}_{M}[m]\simeq a_{M}^{!}({\mathbb C}_{\{pt\}})$
and
${\mathbb C}_{N}[n]\simeq a_{N}^{!}({\mathbb C}_{\{pt\}})$.
Setting
$\iota_{N}\colon N\to M$,
we have
$a_{M}\circ\iota_{N}=a_{N}$.
Therefore, 
by Lemma \ref{ff},
we have locally
the following isomorphisms as an object of ${\rm D}^{b}({\mathbb C}_{M})$:
\begin{eqnarray*}
{\mathbb R}\Gamma_{N}({\mathbb C}_{M})[m]
&\simeq&
{\mathbb R}\Gamma_{N}(a_{M}^{!}({\mathbb C}_{\{pt\}}))
\\
&\simeq&
\iota_{N*} \iota_{N}^{!}a_{M}^{!}({\mathbb C}_{\{pt\}})\\
& \simeq& \iota_{N*} (a_{M}\circ\iota_{N})^{!}({\mathbb C}_{\{pt\}})\\
 &\simeq& \iota_{N*}a_{N}^{!}({\mathbb C}_{\{pt\}})\\
 &\simeq& \iota_{N*}{\mathbb C}_{N}[n].
 \end{eqnarray*}
By applying the exact functor
$\iota_{M*}$,
where $\iota_{M}\colon M\hookrightarrow L$
is a closed embedding,
we have the desired result.
\end{proof}

Recall that
${\mathcal S}\in{\rm Mod}({\mathbb C}_{X})$
is 
called
a locally constant sheaf of finite rank on $X$
if 
for any $x\in X$,
there exists
$l\in{\mathbb Z}_{\geq 0}$
such that 
we have
${\mathcal S}\simeq{\mathbb C}^{l}_{X}$
on a sufficiently small open neighborhood of $x$
in $X$.
For the convenience, we shall often use the following terminology:
\begin{Definition}
	Let ${\mathcal S}\in {\rm Mod}({\mathbb C}_{X})$ and $\iota_{K}\colon K \hookrightarrow X$ a closed subset.
	Then,
	we call ${\mathcal S}$ 
	a
	\textit{locally constant sheaf of finite rank supported in $K$}
	if $\operatorname{supp}({\mathcal S})\subset K$
	and
	$\iota_{K}^{-1}{\mathcal S}$
	is a locally constant sheaf of finite rank on $K$.
\end{Definition}

We need the following two lemmas,
which show that the dimension of the space of the global sections of a locally constant sheaf supported in $K$ is bounded by the dimension of its stalk.

\begin{Lemma}
\label{Gamma<x}
	Let ${\mathcal S}$ be a locally constant sheaf on a connected topological space $X$.
	Then, for any $x\in X$,
	we have
	$$
	\dim \Gamma (X; {\mathcal S})\leq 
	\dim{\mathcal S}_{x},
	$$
	where ${\mathcal S}_{x}$ is the germ of ${\mathcal S}$ at $x$.
\end{Lemma}
\begin{proof}
	It is clear that the natural map 
	$\Gamma (X; {\mathcal S})\to 
	{\mathcal S}_{x}$
	is injective.
\end{proof}

\begin{Lemma}
\label{<ks}
	Let $K\subset X$ be a connected closed subset,
	${\mathcal S}\in {\rm Mod}({\mathbb C}_{X})$ a locally constant sheaf supported in $K$.
	Then, for any $x\in K$, we have 
	$$
	\dim \Gamma(X; {\mathcal S}) \leq 
	\dim{\mathcal S}_{x}.
	$$
\end{Lemma}
\begin{proof}
	By Lemma \ref{gammaKS},
	we have $\Gamma(X; {\mathcal S})=\Gamma(K; \iota_{K}^{-1}{\mathcal S})$.
	Thus,
	the lemma follows from Lemma \ref{Gamma<x}.
\end{proof}

\subsection{${\mathcal D}_{X}$-module}
In this subsection,
we briefly review some properties of ${\mathcal D}$-modules,
which is used in later sections.
We mainly refer to
\cite{K},
\cite{Kas},
and
\cite{SKK}.

Let $X$ be a
$d_{X}$-dimensional complex manifold.
We write 
${\mathcal O}_{X}$ and ${\mathcal D}_{X}$
for
the sheaves of holomorphic functions
and
holomorphic differential operators 
with finite order,
respectively,
on $X$.
For a closed $d_{Y}$-dimensional submanifold $Y\subset X$,
recall that the sheaf ${\mathcal B}_{Y|X}$ is defined by
\begin{align}
{\mathcal B}_{Y|X}:={\mathbb R}\Gamma_{Y}({\mathcal O}_{X})[d_{X}-d_{Y}].
\label{BYX}
\end{align}
In particular,
we have
${\mathcal B}_{X|X}={\mathcal O}_{X}$
in the case $Y=X$.
Note that
the complex
${\mathbb R}\Gamma_{Y}({\mathcal O}_{X})[d_{X}-d_{Y}]$
is concentrated at degree zero
\cite[Prop\:6.2.2]{Sato-II}.
Namely,
we have
${\mathbb R}^{k}\Gamma_{Y}({\mathcal O}_{X})
=0$
for
$k\neq d_{X}-d_{Y}$.

If $X$ is a complexification of some real analytic manifold $M$,
the sheaf
${\mathcal B}_{M}$ of Sato's hyperfunctions is defined by
\begin{align}
{\mathcal B}_{M}:=
{\mathbb R}\Gamma_{M}({\mathcal O}_{X})[d_{X}]
\otimes_{{\mathbb Z}_{M}}
or_{M},
\label{BM}
\end{align}
where 
$or_{M}$ is the orientation sheaf.
Note that the complex
${\mathbb R}\Gamma_{M}({\mathcal O}_{X})[d_{X}]$
is concentrated at degree zero
\cite[Prop\:7.2.2]{Sato-II}
and that
${\mathcal B}_{M}|_{U}={\mathcal B}_{U}$
for an open subset $U\subset M$.

Let
${\mathcal F}$
be the order filtration of
$ {\mathcal D}_{X}$.
The associated graded ring
$gr_{\mathcal F}({\mathcal D}_{X})$
of ${\mathcal D}_{X}$ with respect to ${\mathcal F}$
is defined
by
$gr_{\mathcal F}({\mathcal D}_{X}):=\bigoplus_{j\in{\mathbb Z}_{\geq 0}}{\mathcal F}_{j}({\mathcal D}_{X})/{\mathcal F}_{j-1}({\mathcal D}_{X})$.
Note that
there exists an injection 
$\pi^{-1}(gr_{\mathcal F}({\mathcal D}_{X}))\hookrightarrow{\mathcal O}_{T^{*}X}$,
where
$\pi\colon T^{*}X\to X$
is the natural projection.
For a $gr_{\mathcal F}({\mathcal D}_{X})$-module $M$,
we set
\begin{align}
F(M):=
{\mathcal O}_{T^{*}X}
\otimes_{\pi^{-1}(gr_{\mathcal F}({\mathcal D}_{X}))}
\pi^{-1}(M).
\label{Ffunctor}
\end{align}
Then
$F$ defines a functor from the category of 
$gr_{\mathcal F}({\mathcal D}_{X})$-modules to
that of 
${\mathcal O}_{T^{*}X}$-modules.
By the exactness of the inverse image functor
and
the right exactness of the tensor product,
$F$
is right exact.

Let ${\mathfrak M}$ be a coherent ${\mathcal D}_{X}$-module.
The characteristic variety 
${\rm Ch}({\mathfrak M})\subset T^{*}X$ of
${\mathfrak M}$
is defined
by
the support of
a
${\mathcal O}_{T^{*}X}$-module $F(gr_{{\mathcal F}_{\mathfrak M}}({\mathfrak M}))$,
where
${\mathcal F}_{\mathfrak M}$
is a good filtration of 
${\mathfrak M}$.
This definition does not depend on the choice of a good filtration.
For an irreducible closed analytic subset $V$ of
$T^{*}X$,
we write
${\rm mult}^{{\mathcal D}_{X}}_{V}({\mathfrak M})$
for
the multiplicity of 
${\mathfrak M}$
along $V$ \cite[Prop.\:1.8.2]{Bjork},
which is given by
$$
{\rm mult}_{V}^{{\mathcal D}_{X}}({\mathfrak M})
:=
{\rm mult}_{V}^{{\mathcal O}_{T^{*}X}}
(
F(gr_{{\mathcal F}_{\mathfrak M}}({\mathfrak M}))).
$$
Here,
${\rm mult}_{V}^{{\mathcal O}_{T^{*}X}}(
F(gr_{{\mathcal F}_{\mathfrak M}}({\mathfrak M}))$
is the multiplicity 
of
an
${\mathcal O}_{T^{*}X}$-module
$F(gr_{{\mathcal F}_{\mathfrak M}}({\mathfrak M}))$
along
$V$
(See \cite[Sect.\:1.4 in Appx.\:V]{Bjork},
for example).
Note that if $V$ is not contained in ${\rm Ch}({\mathfrak M})$,
we have ${\rm mult}_{V}^{{\mathcal D}_{X}}({\mathfrak M})=0$.

\begin{Remark}
\label{mult-infty}
	It is known that if two holonomic ${\mathcal D}_{X}$-modules ${\mathfrak M}$ and ${\mathfrak M}'$ satisfy
	${\mathcal D}_{X}^{\infty}\otimes_{{\mathcal D}_{X}} {\mathfrak M}\simeq 
	{\mathcal D}_{X}^{\infty}\otimes_{{\mathcal D}_{X}} {\mathfrak M}'$,
	we have ${\rm mult}_{V}^{{\mathcal D}_{X}}({\mathfrak M}) ={\rm mult}_{V}^{{\mathcal D}_{X}}({\mathfrak M'})$,
	where ${\mathcal D}_{X}^{\infty}$ is 
	the sheaf of 
holomorphic differential operators 
of infinite order.
This is because the multiplicity of ${\mathfrak M}$ is determined by the perverse sheaf ${\mathbb R}{\mathcal Hom}_{{\mathcal D}_{X}}({\mathfrak M},{\mathcal O}_{X})\simeq {\mathbb R}{\mathcal Hom}_{{\mathcal D}_{X}^{\infty}}({\mathcal D}_{X}^{\infty}\otimes_{{\mathcal D}_{X}} {\mathfrak M},{\mathcal O}_{X})$
(see \cite[Sect.\:8.2]{K-index} or \cite[(7.23)]{SV}).
We use this fact in Appendix to prove
Corollary \ref{CorRC}, given below.
\end{Remark}

A coherent ${\mathcal D}_{X}$-module ${\mathfrak M}$
is called
holonomic 
if
the characteristic variety 
${\rm Ch}({\mathfrak M})\subset T^{*}X$
of
${\mathfrak M}$
is $d_{X}$-dimensional
or
${\mathfrak M}=0$.
We note that
if a coherent ${\mathcal D}_{X}$-module
${\mathfrak M}$ is nonzero,
we have
$\dim {\rm Ch}({\mathfrak M})\geq d_{X}$
by
the involutivity of
the characteristic variety 
\cite[Thm.\:5.3.2]{SKK}.
A stratification
$X=\bigsqcup_{\alpha\in A}X_{\alpha}$ of $X$
is called regular with respect to a holonomic ${\mathcal D}_{X}$-module ${\mathfrak M}$
if
the stratification
$X=\bigsqcup_{\alpha\in A}X_{\alpha}$
satisfies the 
regularity conditions of H. Whitney
\cite[(a),(b) in Sect.\:19]{Whitney}
and
${\rm Ch}({\mathfrak M})$
is contained in
the union 
$\bigsqcup_{\alpha\in A}T^{*}_{X_{\alpha}}X$
of the conormal bundles 
of each stratum
\cite[Def.\:(3.4)]{K75}.
For any holonomic ${\mathcal D}_{X}$-module
${\mathfrak M}$,
there exists
a regular stratification of $X$
with respect to
${\mathfrak M}$
\cite[Lem.\:(3.2)]{K75}.

We use Kashiwara's
constructibility theorem of holonomic ${\mathcal D}_{X}$-modules:
\begin{Fact}[{\cite[Thms.\:(3.5) and (3.7)]{K75}}]
\label{holonomic}
Let 
${\mathfrak M}$
be a holonomic ${\mathcal D}_{X}$-module
and
$X=\bigsqcup_{\alpha\in A}X_{\alpha}$
a regular stratification of $X$
with respect to
${\mathfrak M}$.
Then, 
${\mathbb R}^{k}{\mathcal Hom}_{{\mathcal D}_{X}}({\mathfrak M},{\mathcal O}_{X})|_{X_{\alpha}}$
is a locally constant sheaf of finite rank on $X_{\alpha}$
for any
$\alpha\in A$
and
any $k\in{\mathbb Z}$.
Moreover,
if 
$Y:=\bigsqcup_{\beta \in B}X_{\beta}$
is a closed subset of $X$
for some subset $B\subset A$,
then
${\mathbb R}^{k}{\mathcal Hom}_{{\mathcal D}_{X}}({\mathfrak M},{\mathcal B}_{Y|X})|_{X_{\alpha}}$
is also
a locally constant sheaf of finite rank on $X_{\alpha}$
for any 
$\alpha\in A$
and
any $k\in{\mathbb Z}$.
\end{Fact}

\begin{Corollary}
\label{CorRC}
Let
$X=\bigsqcup_{\alpha\in A}X_{\alpha}$
be
a regular stratification of $X$
with respect to
a holonomic
${\mathcal D}_{X}$-module.
Suppose that
$X_{\alpha}$ 
is closed in $X$.
Then,
$
{\mathbb R}^{k}
{\mathcal Hom}_{{\mathcal D}_{X}}
({\mathfrak M},
{\mathcal B}_{X_{\alpha} | X})
$
is a locally constant sheaf of finite rank supported in $X_{\alpha}$
for any $k\in{\mathbb Z}$.
Moreover,
the rank of 
$
{\mathcal Hom}_{{\mathcal D}_{X}}
({\mathfrak M},
{\mathcal B}_{X_{\alpha} | X})
$
is not greater than
$\operatorname{mult}_{T_{X_{\alpha}}^{*}X}^{{\mathcal D}_{X}}({\mathfrak M})$.
\end{Corollary}
\begin{proof}
The assertion for $
{\mathbb R}^{k}
{\mathcal Hom}_{{\mathcal D}_{X}}
({\mathfrak M},
{\mathcal B}_{X_{\alpha} | X})
$ is obvious by Fact \ref{holonomic}.
The assertion for the rank of 
${\mathcal Hom}_{{\mathcal D}_{X}}
({\mathfrak M},
{\mathcal B}_{X_{\alpha} | X})$
is proved in Appendix
\ref{Reg-mult} 
because we use a different method from this section.
\end{proof}

\begin{Remark}
Although we use the theory of regular holonomic ${\mathcal D}_{X}$-modules in this article (see Appendix \ref{Reg-mult}),
	Corollary \ref{CorRC} can be proved by using the theory of ${\mathcal E}^{\mathbb R}_{X}$-modules, see \cite[Thm.\:3.2.42]{K}.
	Moreover, if the support of ${\mathfrak M}$ is equal to $X_{\alpha}$ (and with the assumption of Corollary \ref{CorRC}), then the rank of ${\mathcal Hom}_{{\mathcal D}_{X}}({\mathfrak M},{\mathcal B}_{X_{\alpha}|X})$ is equal to $\operatorname{mult}_{T_{X_{\alpha}}^{*}X}^{{\mathcal D}_{X}}({\mathfrak M})$ by \cite[Prop.\:3.9]{K75}.
	\end{Remark}

\section{Upper bound of the dimensions of the Sato hyperfunction solution spaces
}
\label{Holonomy-dimension}
In this section,
we give
an
upper bound of the dimensions of the Sato hyperfunction solution spaces.
First,
we estimate the dimension of the space of solutions
whose supports are contained in 
one closed stratum $X_{\alpha}$.
\begin{Proposition}
\label{holonomy-dimension}
Let $M$ be a real analytic manifold,
$X$ a complexification of $M$,
and
$X=\bigsqcup_{\alpha\in A}X_{\alpha}$
a regular stratification of $X$
with respect to 
a holonomic ${\mathcal D}_{X}$-module ${\mathfrak M}$.
Suppose that 
$X_{\alpha}$ is a closed connected subset $X$ 
and that a subset
$M_{\alpha}'$ of $M\cap X_{\alpha}$
is a closed submanifold of $X_{\alpha}$.
Then,
${\mathcal Hom}_{{\mathcal D}_{X}}
(
{\mathfrak M},
\Gamma_{M_{\alpha}'}(
{\mathcal B}_{M}
))$
is a locally constant sheaf
supported in $M_{\alpha}'$
whose rank is not greater than
$\operatorname{mult}_{T_{X_{\alpha}}^{*}X}^{{\mathcal D}_{X}}({\mathfrak M})$.
\end{Proposition}

Before proving,
we recall the Grothendieck spectral sequence.
\begin{Fact}[{\cite[Thm.\:2.4.1]{Grothendieck-Tohoku}}]
\label{spectral}
Let
${\mathcal C},{\mathcal C}',{\mathcal C}''$
be abelian categories,
$F\colon {\mathcal C}\to{\mathcal C}'$
and
$G\colon{\mathcal C}'\to{\mathcal C}''$
left exact functors.
Suppose that
$F$
takes injective objects of ${\mathcal C}$
to injective objects of ${\mathcal C}'$.
Then for any $A\in {\mathcal C}$,
we have a spectral sequence 
$$
E^{p,q}_{2}={\mathbb R}^{p}G\circ{\mathbb R}^{q}F(A)
\Rightarrow
{\mathbb R}^{p+q}(G\circ F)(A).
$$
\end{Fact}

\begin{proof}[Proof of Proposition \ref{holonomy-dimension}.]
It is clear that the support of 
${\mathcal Hom}_{{\mathcal D}_{X}}
(
{\mathfrak M},
\Gamma_{M_{\alpha}'}(
{\mathcal B}_{M}
))$
is contained in $M_{\alpha}'$.
Therefore,
it is sufficient to prove that
for any $x_{\alpha} \in M_{\alpha}'$,
there exists an isomorphism 
$$
{\mathcal Hom}_{{\mathcal D}_{X}}
(
{\mathfrak M},
\Gamma_{M_{\alpha}'}(
{\mathcal B}_{M}
))
\simeq 
{\mathbb C}_{M_{\alpha}'}^{l}
$$
on a sufficiently small open neighborhood $U_{\alpha}\subset X$ of
$x_{\alpha}\in M_{\alpha}'$
satisfying $l\leq 
\operatorname{mult}_{T_{X_{\alpha}}^{*}X}^{{\mathcal D}_{X}}({\mathfrak M})$.
Corollary \ref{CorRC} implies that
there exists an isomorphism
\begin{eqnarray}
{\mathbb R}^{k}{\mathcal Hom}_{{\mathcal D}_{X}}
(
{\mathfrak M},
{\mathcal B}_{X_{\alpha}|X}
)
\simeq
{\mathbb C}^{l_{k}}_{X_{\alpha}}
\label{RC}
\end{eqnarray}
for some $l_{k}\in{\mathbb Z}_{\geq 0}$
on a sufficiently small open neighborhood $U_{\alpha}$ of
$x_{\alpha}\in X_{\alpha}$
satisfying 
\begin{eqnarray}
\label{l0dim}
l_{0}\leq 
\operatorname{mult}_{T_{X_{\alpha}}^{*}X}^{{\mathcal D}_{X}}({\mathfrak M}).
\end{eqnarray}
Moreover,
we have a chain of isomorphisms on $U_{\alpha}$
\begin{align}
{\mathbb R}\Gamma_{M_{\alpha}'}{\mathbb R}{\mathcal Hom}_{{\mathcal D}_{X}}({\mathfrak M},{\mathcal B}_{X_{\alpha}|X})[d_{X_{\alpha}}]
&\simeq
{\mathbb R}{\mathcal Hom}_{{\mathcal D}_{X}}
({\mathfrak M},{\mathbb R}\Gamma_{M_{\alpha}'}({\mathcal B}_{X_{\alpha}|X}))[d_{X_{\alpha}}]\nonumber\\
&\simeq
{\mathbb R}{\mathcal Hom}_{{\mathcal D}_{X}}({\mathfrak M},{\mathbb R}\Gamma_{M_{\alpha}'}{\mathbb R}\Gamma_{X_{\alpha}}({\mathcal O}_{X}))[d_{X}]\nonumber\\
&\simeq
{\mathbb R}{\mathcal Hom}_{{\mathcal D}_{X}}({\mathfrak M},{\mathbb R}\Gamma_{M_{\alpha}'}{\mathbb R}\Gamma_{M}({\mathcal O}_{X}))[d_{X}]\nonumber\\
&\simeq
{\mathbb R}{\mathcal Hom}_{{\mathcal D}_{X}}({\mathfrak M},{\mathbb R}\Gamma_{M_{\alpha}'}({\mathcal B}_{M})).
\label{isoms}
\end{align}
Here,  
we omit the orientation sheaf
because the isomorphisms are local.
For the first isomorphism,
see \cite[Thm.\:7.9 in Appx.\:II]{Bjork}.
The second and fourth isomorphisms follow from
\eqref{BYX} and
\eqref{BM},
respectively.
In the third isomorphism,
we have used 
$M_{\alpha}'\cap X_{\alpha}=M_{\alpha}'=M_{\alpha}'\cap M$
and
Lemma \ref{KK'}.
By the definition,
we have
\begin{align}
H^{0}
(
{\mathbb R}{\mathcal Hom}_{{\mathcal D}_{X}}({\mathfrak M},{\mathbb R}\Gamma_{M_{\alpha}'}({\mathcal B}_{M}))
)
\simeq
{\mathcal Hom}_{{\mathcal D}_{X}}({\mathfrak M},\Gamma_{M_{\alpha}'}({\mathcal B}_{M})).
\label{R0}
\end{align}
Here,
we write
$
H^{0}
$
for the $0$-th cohomology functor.
By Fact
\ref{spectral},
there exists a
Grothendieck spectral sequence
\begin{equation}
E^{p,q}_{2}\simeq 
{\mathbb R}^{p}\Gamma_{M_{\alpha}'}
({\mathbb R}^{q}
{\mathcal Hom}_{{\mathcal D}_{X}}({\mathfrak M},{\mathcal B}_{X_{\alpha}|X})
)
\Rightarrow
H^{p+q}
(
{\mathbb R}\Gamma_{M_{\alpha}'}{\mathbb R}
{\mathcal Hom}_{{\mathcal D}_{X}}({\mathfrak M},{\mathcal B}_{X_{\alpha}|X})
).
\label{Grothendieck}
\end{equation}
Note that ${\mathcal Hom}_{{\mathcal D}_{X}}({\mathfrak M},*)$
takes injective objects 
of the category of ${\mathcal D}_{X}$-modules
to injective objects of ${\rm Mod}({\mathbb C}_{X})$
(\cite[Prop.\:6.21 in Appx.\:II]{Bjork}
or
\cite[Prop.\:2.4.6 (vii)]{Kas}).
Lemma \ref{MN} 
and
\eqref{RC}
imply
\begin{eqnarray}
E^{p,q}_{2}\simeq
{\mathbb R}^{p}\Gamma_{M_{\alpha}'}
({\mathbb C}^{l_{q}}_{X_{\alpha}})\simeq 
\begin{cases}
  {\mathbb C}_{M_{\alpha}'}^{l_{q}}& (p=2d_{X_{\alpha}}-\dim M_{\alpha}'), \\
    0 & (otherwise).
  \end{cases}\label{Epq}
\end{eqnarray}
Because
$X$ is a complexification of $M$,
we have
$\dim M_{\alpha}'\leq d_{X_{\alpha}}$.
Thus,
\eqref{isoms},
\eqref{R0},
\eqref{Grothendieck}
and
\eqref{Epq}
imply
\begin{eqnarray}
{\mathcal Hom}_{{\mathcal D}_{X}}({\mathfrak M},\Gamma_{M_{\alpha}'}({\mathcal B}_{M}))
\simeq
\begin{cases}
 {\mathbb C}_{M_{\alpha}'}^{l_{0}}& (\dim M_{\alpha}'=d_{X_{\alpha}}), \\
    0 & (\dim M_{\alpha}'\neq d_{X_{\alpha}})
  \end{cases}\label{HomC}
\end{eqnarray}
on a sufficiently small open neighborhood of $x_{\alpha}\in M_{\alpha}'$.
Therefore,
this completes the proof because
$l_{0}
\leq 
\operatorname{mult}_{T_{X_{\alpha}}^{*}X}^{{\mathcal D}_{X}}({\mathfrak M})
$
by
\eqref{l0dim}.
\end{proof}

We want to apply Proposition \ref{holonomy-dimension} to 
the case $M_{\alpha}'=M\cap X_{\alpha}$.
However, it is impossible
because
$M_{\alpha}:=M\cap X_{\alpha}$ may have a singular point
in general.
In order to overcome this,
we consider a stratification $M_{\alpha}=\bigsqcup_{j=1}^{J}M_{\alpha}^{(j)}$
such that $M_{\alpha}^{(j)}$ is a submanifold of $X_{\alpha}$.
In fact,
such a stratification locally exists 
by the theory of semianalytic sets.
\begin{Lemma}
\label{semiMalpha}
	Let $M$ be a real analytic manifold,
$X$ a complexification of $M$,
and
$X=\bigsqcup_{\alpha\in A}X_{\alpha}$
a regular stratification of $X$.
Set $M_{\alpha}:=X_{\alpha}\cap M$
for any $\alpha\in A$.
Then,
for any $\alpha\in A$ and
any relatively compact semianalytic subset $U_{\mathbb C}$ of $X$ (as a real analytic manifold),
there exists a 
stratification
$M_{\alpha}\cap U_{\mathbb C}=\bigsqcup_{j=1}^{J_{\alpha}} M_{\alpha}^{(j)}$
such that $J_{\alpha}<\infty$
and
$M_{\alpha}^{(j+1)}$ is a closed connected submanifold
of $(X_{\alpha}\cap U_{\mathbb C})\backslash \bigsqcup_{i=1}^{j}M_{\alpha}^{(i)}$.
\end{Lemma}
\begin{proof}
	Note that $M_{\alpha}$ is a semianalytic subset
	because $X_{\alpha}$ is a semianalytic subset of $X$.
	Thus,
	this lemma is a direct consequence of \cite[Prop.\:2.10]{semifinite}.
\end{proof}

\begin{Corollary}
\label{Cor-dim-Gamma}
Let $M$ be a real analytic manifold,
$X$ a complexification of $M$,
and
$X=\bigsqcup_{\alpha\in A}X_{\alpha}$
a regular stratification of $X$
with respect to 
a holonomic ${\mathcal D}_{X}$-module ${\mathfrak M}$.
Suppose that 
$X_{\alpha}$ is a closed connected subset of $X$
and set
$M_{\alpha}:=M\cap X_{\alpha}$.
Let $U_{\mathbb C}$ be a relatively compact semianalytic open subset of $X$
and $M_{\alpha}\cap U_{\mathbb C}=\bigsqcup_{j=1}^{J_{\alpha}} M_{\alpha}^{(j)}$
a stratification given in Lemma \ref{semiMalpha}.
Then,
we have
\begin{eqnarray*}
	\dim 
	\Gamma(U;
	{\mathcal Hom}_{{\mathcal D}_{X}}
(
{\mathfrak M},
\Gamma_{M_{\alpha}}(
{\mathcal B}_{M}
)))
\leq 
J_{\alpha}\cdot 
\operatorname{mult}_{T_{X_{\alpha}}^{*}X}^{{\mathcal D}_{X}}({\mathfrak M})
\end{eqnarray*}
where $U:=U_{\mathbb C}\cap M$.
\end{Corollary}
\begin{proof}
	Because we only consider the space of sections over $U$,
	we may assume that $X=U_{\mathbb C}$
	and $M=U$.
	Thus,
	$M_{\alpha}^{(j+1)}$ is a closed connected submanifold of $X_{\alpha}\backslash \bigsqcup_{i=1}^{j}M_{\alpha}^{(i)}$.
	For simplicity,
we set
\begin{eqnarray*}
X^{j}:=X\backslash \bigsqcup_{i=1}^{j}M_{\alpha}^{(i)},&&
X_{\alpha}^{j}:=X_{\alpha}\backslash \bigsqcup_{i=1}^{j}M_{\alpha}^{(i)},\\
M^{j}:=M\backslash \bigsqcup_{i=1}^{j}M_{\alpha}^{(i)},
&&
M^{j}_{\alpha}:=M_{\alpha}\backslash \bigsqcup_{i=1}^{j}M_{\alpha}^{(i)}.
\end{eqnarray*}
We use the convention $X^{0}:=X$, $X^{0}_{\alpha}:=X_{\alpha}$, $M^{0}:=M$, and $M_{\alpha}^{0}:=M_{\alpha}$.
In this notation, 
we have
\begin{eqnarray*}
\begin{array}{cccccc}
X                    &
\overset{open}{\supset} & 
X^{j}    &   
\overset{closed}{\supset}  & 
X^{j}_{\alpha}            
\\[4pt]
\rotatebox{90}{$\subset$} &
& 
\rotatebox{90}{$\subset$} &
&
\rotatebox{90}{$\subset$}
\\
M                    & 
\overset{open}{\supset}     & 
M^{j} = M\cap X^{j}     &
\overset{closed}{\supset}  &
M^{j}_{\alpha} = M^{j} \cap X_{\alpha}^{j}
&
\overset{closed}{\supset}
M^{(j+1)}_{\alpha}.
\end{array}
\end{eqnarray*}
Note that
$M_{\alpha}^{(j+1)} \subset M^{j}\cap X_{\alpha}^{j}$ is a closed connected submanifold of $X_{\alpha}^{j}$.

We want to prove the corollary by using the filtration by support (cf.\:\cite[(6.10)]{KS}) and Proposition \ref{holonomy-dimension}.
First, we treat the case of $\Gamma_{M_{\alpha}^{(1)}}({\mathcal B}_{M})$.
	Because $M_{\alpha}^{(1)}$ is closed in $M_{\alpha}=M_{\alpha}^{0}$,
	we have an exact sequence
\begin{eqnarray}
0\to
\Gamma_{M_{\alpha}^{(1)}}({\mathcal B}_{M})
\to
\Gamma_{M_{\alpha}}({\mathcal B}_{M})
\to
\Gamma_{M_{\alpha}^{1}}({\mathcal B}_{M})
\label{exactBBB}
\end{eqnarray}
by Lemma \ref{ZZZ}.
Applying the left exact functor
$\Gamma(M;
{\mathcal Hom}_{{\mathcal D}_{X}}(
{\mathfrak M},
\cdot))
$, we have
\begin{alignat*}{1}
0
&\to
\Gamma(M;
{\mathcal Hom}_{{\mathcal D}_{X}}(
{\mathfrak M},
\Gamma_{M_{\alpha}^{(1)}}
(
{\mathcal B}_{M}
)
)
)
\\
&\to
\Gamma(M;
{\mathcal Hom}_{{\mathcal D}_{X}}(
{\mathfrak M},
\Gamma_{M_{\alpha}}({\mathcal B}_{M})
))
\to
\Gamma(M;
{\mathcal Hom}_{{\mathcal D}_{X}}(
{\mathfrak M},
\Gamma_{M_{\alpha}^{1}}
(
{\mathcal B}_{M}
)
)),
\end{alignat*}
which is exact.
Therefore,
we have
\begin{eqnarray}
&&\dim
\Gamma(M;
{\mathcal Hom}_{{\mathcal D}_{X}}(
{\mathfrak M},
\Gamma_{M_{\alpha}}({\mathcal B}_{M})
))\label{dim0mult}\\
&\leq  &
\dim\Gamma(M;
{\mathcal Hom}_{{\mathcal D}_{X}}(
{\mathfrak M},
\Gamma_{M_{\alpha}^{(1)}}
(
{\mathcal B}_{M}
)
))+
\dim\Gamma(M;
{\mathcal Hom}_{{\mathcal D}_{X}}(
{\mathfrak M},
\Gamma_{M_{\alpha}^{1}}
(
{\mathcal B}_{M}
)
)).\nonumber
\end{eqnarray}
Because $M_{\alpha}^{(1)} \subset M\cap X_{\alpha}$ is a closed connected submanifold of $X_{\alpha}$,
we have
	\begin{eqnarray}
	\dim 
	\Gamma(M;
	{\mathcal Hom}_{{\mathcal D}_{X}}
(
{\mathfrak M},
\Gamma_{M_{\alpha}^{(1)}}(
{\mathcal B}_{M}
)))
\leq 
\operatorname{mult}_{T_{X_{\alpha}}^{*}X}^{{\mathcal D}_{X}}({\mathfrak M})
\label{dimGamma<mult}
	\end{eqnarray}
	by Lemma \ref{<ks}
	and
	Proposition \ref{holonomy-dimension}.
By \eqref{dim0mult} and \eqref{dimGamma<mult},
we have
\begin{eqnarray}
&&\dim
\Gamma(M;
{\mathcal Hom}_{{\mathcal D}_{X}}(
{\mathfrak M},
\Gamma_{M_{\alpha}}({\mathcal B}_{M})
))\label{dim0mult2} \\
&\leq  &
\operatorname{mult}_{T_{X_{\alpha}}^{*}X}^{{\mathcal D}_{X}}({\mathfrak M})
+
\dim\Gamma(M;
{\mathcal Hom}_{{\mathcal D}_{X}}(
{\mathfrak M},
\Gamma_{M_{\alpha}^{1}}
(
{\mathcal B}_{M}
)
)).\nonumber
\end{eqnarray}

We want to apply the same argument to the last term of \eqref{dim0mult2}.
For this end, we rewrite it.
We have a chain of isomorphisms
\begin{eqnarray}
&&
\Gamma(M;
{\mathcal Hom}_{{\mathcal D}_{X}}(
{\mathfrak M},
\Gamma_{M^{1}_{\alpha}}
(
{\mathcal B}_{M}
)
))
\nonumber\\
&\simeq &
\Gamma(M;
{\mathcal Hom}_{{\mathcal D}_{X}}(
{\mathfrak M},
\Gamma_{M^{1}}
\Gamma_{M^{1}_{\alpha}}
(
{\mathcal B}_{M}
)
))
\nonumber
\\
&\simeq &
\Gamma(M;
\Gamma_{M^{1}}
{\mathcal Hom}_{{\mathcal D}_{X}}(
{\mathfrak M},
\Gamma_{M^{1}_{\alpha}}
(
{\mathcal B}_{M}
)
))
\nonumber
\\
&\simeq &
\Gamma(M^{1};
{\mathcal Hom}_{{\mathcal D}_{X}}(
{\mathfrak M},
\Gamma_{{M}^{1}_{\alpha}}
(
{\mathcal B}_{M}
)
))
\nonumber
\\
&\simeq &
\Gamma(M^{1};
{\mathcal Hom}_{{\mathcal D}_{X^{1}}}(
{\mathfrak M}|_{X^{1}},
\Gamma_{M^{1}_{\alpha}}
(
{\mathcal B}_{M^{1}}
)
)),
\label{Gamma=Gamma}
\end{eqnarray}
where ${\mathfrak M}|_{X^{1}}$ is the restriction of ${\mathfrak M}$ to $X^{1}$.
In the first isomorphism,
we have used 
$M^{1}_{\alpha}=M^{1}\cap M^{1}_{\alpha}$
and
\cite[Prop.\:2.3.9 (ii)]{Kas}.
For the second isomorphism,
see \cite[(2.3.18)]{Kas}.
The third and last isomorphisms follow from the definitions.

Recall that
$M_{\alpha}^{(2)}$ is closed in $M_{\alpha}^{1}$.
Thus,
we have an exact sequence
\begin{eqnarray*}
0\to
\Gamma_{M_{\alpha}^{(2)}}({\mathcal B}_{M^{1}})
\to
\Gamma_{M_{\alpha}^{1}}({\mathcal B}_{M^{1}})
\to
\Gamma_{M_{\alpha}^{2}}({\mathcal B}_{M^{1}})
\end{eqnarray*}
by Lemma \ref{ZZZ}.
Applying the left exact functor
$\Gamma(M^{1};
{\mathcal Hom}_{{\mathcal D}_{X^{1}}}(
{\mathfrak M}|_{X^{1}},
\cdot))$,
we have
\begin{eqnarray}
&&\dim
\Gamma(M^{1};
{\mathcal Hom}_{{\mathcal D}_{X^{1}}}(
{\mathfrak M}|_{X^{1}},
\Gamma_{M^{1}_{\alpha}}
(
{\mathcal B}_{M^{1}}
)
))
\label{dim<dim+dimGamma} \\
&\leq &
\dim\Gamma(M^{1};
{\mathcal Hom}_{{\mathcal D}_{X^{1}}}(
{\mathfrak M}|_{X^{1}},
\Gamma_{M_{\alpha}^{2}}
(
{\mathcal B}_{M^{1}}
)
))
\nonumber\\
&&+
\dim
\Gamma(M^{1};
{\mathcal Hom}_{{\mathcal D}_{X^{1}}}(
{\mathfrak M}|_{X^{1}},
\Gamma_{M^{(2)}_{\alpha}}
(
{\mathcal B}_{M^{1}}
)
)).
\nonumber
\end{eqnarray}

We want to apply Proposition \ref{holonomy-dimension} to the last term of \eqref{dim<dim+dimGamma}.
For this end, we write $X_{\alpha}^{1}=\bigsqcup_{k\in K_{\alpha}}X^{1}_{\alpha,k}$ for the connected component decomposition of $X_{\alpha}^{1}$
and take $k_{0}\in K_{\alpha}$ such that $X_{\alpha,k_{0}}^{1}$ is the connected component of $X_{\alpha}^{1}$ containing $M_{\alpha}^{(2)}$.
We shall check the assumption of Proposition \ref{holonomy-dimension}.
It is clear that $M^{1}$ is a real analytic manifold,
$X^{1}$ is its complexification
and
$$X^{1}=\left(\bigsqcup_{\beta\neq \alpha}X_{\beta}\right)\sqcup
\left(\bigsqcup_{k\in K_{\alpha}}X^{1}_{\alpha,k}\right)$$
is a regular stratification of $X^{1}$
with respect to ${\mathfrak M}|_{X^{1}}$.
Moreover, 
$X_{\alpha,k_{0}}^{1}$ is a closed connected subset of $X^{1}$ because $X_{\alpha,k_{0}}^{1}$ is the connected component of a closed subset $X_{\alpha}^{1}\subset X^{1}$.
Moreover, $M_{\alpha}^{(2)}\subset M^{1}\cap X_{\alpha,k_{0}}^{1}$ is a closed connected submanifold of $X_{\alpha,k_{0}}^{1}$  by definition.
Thus,
we have
	\begin{eqnarray}
	\label{dim-mult-D-2}
\dim
\Gamma(M^{1};
{\mathcal Hom}_{{\mathcal D}_{X^{1}}}(
{\mathfrak M}|_{X^{1}},
\Gamma_{M^{(2)}_{\alpha}}
(
{\mathcal B}_{M^{1}}
)
))
\leq 
\operatorname{mult}_{T_{X_{\alpha,k_{0}}^{1}}^{*}X^{1}}^{{\mathcal D}_{X^{1}}}({\mathfrak M}|_{X^{1}})
	\end{eqnarray}
	by Lemma \ref{<ks}
	and
	Proposition \ref{holonomy-dimension}.
Moreover,
\begin{eqnarray}
\operatorname{mult}_{T_{X_{\alpha,k_{0}}^{1}}^{*}X^{1}}^{{\mathcal D}_{X^{1}}}({\mathfrak M}|_{X^{1}})
=
\operatorname{mult}_{T_{X_{\alpha}}^{*}X}^{{\mathcal D}_{X}}({\mathfrak M})
\label{mult=mult10}	
\end{eqnarray}
follows easily from the definition (because the multiplicity of an ${\mathcal O}$-module is defined by the length of it at a generic point, see \cite[Sect.\:2.6]{K} for example).
Thus,
\eqref{dim<dim+dimGamma},
\eqref{dim-mult-D-2}
and
\eqref{mult=mult10}
imply
\begin{eqnarray}
	&&\dim
\Gamma(M^{1};
{\mathcal Hom}_{{\mathcal D}_{X^{1}}}(
{\mathfrak M}|_{X^{1}},
\Gamma_{M^{1}_{\alpha}}
(
{\mathcal B}_{M^{1}}
)
))
\label{dim<dim+dimGamma3} \\
&\leq &
\dim\Gamma(M^{1};
{\mathcal Hom}_{{\mathcal D}_{X^{1}}}(
{\mathfrak M}|_{X^{1}},
\Gamma_{M_{\alpha}^{2}}
(
{\mathcal B}_{M^{1}}
)
))
+
\operatorname{mult}_{T_{X_{\alpha}}^{*}X}^{{\mathcal D}_{X}}({\mathfrak M}).\nonumber
\end{eqnarray}
Thus, 
\eqref{dim0mult2},
\eqref{Gamma=Gamma},
and
\eqref{dim<dim+dimGamma3}
imply
\begin{eqnarray*}
	&&\dim
\Gamma(M;
{\mathcal Hom}_{{\mathcal D}_{X}}(
{\mathfrak M},
\Gamma_{M_{\alpha}}({\mathcal B}_{M})
))\\
\leq  &&
2\operatorname{mult}_{T_{X_{\alpha}}^{*}X}^{{\mathcal D}_{X}}({\mathfrak M})
+
\dim\Gamma(M^{1};
{\mathcal Hom}_{{\mathcal D}_{X^{1}}}(
{\mathfrak M}|_{X^{1}},
\Gamma_{M_{\alpha}^{2}}
(
{\mathcal B}_{M^{1}}
)
)).
\end{eqnarray*}
Repeating the same argument,
we have the corollary.
\end{proof}

Considering the filtration by support
(cf.\:\cite[(6.10)]{KS}),
we get the desired result.

\begin{Proposition}
\label{dim-rela-U}
Let $M$ be a real analytic manifold,
$X$ a complexification of $M$,
and
$X=\bigsqcup_{\alpha\in A}X_{\alpha}$
a regular stratification of $X$
with respect to 
a holonomic ${\mathcal D}_{X}$-module ${\mathfrak M}$ such that each stratum $X_{\alpha}$ is connected.
Let $U$ be a relatively compact semianalytic open subset of $M$
and $J_{\alpha}$ the integer given in Lemma \ref{semiMalpha} for any $\alpha \in A$.
Then,
we have
\begin{eqnarray*}
		\dim 
	\Gamma(U;
	{\mathcal Hom}_{{\mathcal D}_{X}}
(
{\mathfrak M},
{\mathcal B}_{M}
))
\leq 
\sum_{\alpha\in A}
J_{\alpha}\cdot 
\operatorname{mult}_{T_{X_{\alpha}}^{*}X}^{{\mathcal D}_{X}}({\mathfrak M}).
\end{eqnarray*}
\end{Proposition}
\begin{proof}
Take a relatively compact semianalytic open subset $U_{\mathbb C}$ of $X$ satisfying $U_{\mathbb C}\cap M=U$.
Because we only consider the space of sections over $U$,
	we may assume that $X=U_{\mathbb C}$
	and $M=U$.

Because 
$X=\bigsqcup_{\alpha\in A}X_{\alpha}$
is a regular stratification,
there exists
$\alpha_{0}\in A$
such that
$X_{\alpha_{0}}$ is closed in $X$.
Applying Corollary \ref{Cor-dim-Gamma},
we have
\begin{eqnarray}
	\dim 
	\Gamma(U;
	{\mathcal Hom}_{{\mathcal D}_{X}}
(
{\mathfrak M},
\Gamma_{M_{\alpha_{0}}}(
{\mathcal B}_{M}
)))
\leq 
J_{\alpha_{0}}\cdot 
\operatorname{mult}_{T_{X_{\alpha_{0}}}^{*}X}^{{\mathcal D}_{X}}({\mathfrak M}).
\label{dimJa}
\end{eqnarray}
Because $M_{\alpha_{0}}$ is closed in $M$
and $\Gamma_{M}({\mathcal B}_{M})={\mathcal B}_{M}$,
we have an exact sequence
$$
0\to
\Gamma_{M_{\alpha_{0}}}({\mathcal B}_{M})
\to
{\mathcal B}_{M}
\to
\Gamma_{M\backslash M_{\alpha_{0}}}({\mathcal B}_{M})
$$
by Lemma \ref{ZZZ}.
Applying the left exact functor
$\Gamma(U;
{\mathcal Hom}_{{\mathcal D}_{X}}(
{\mathfrak M},
\cdot))
$, we have
\begin{alignat*}{1}
0
&\to
\Gamma(U;
{\mathcal Hom}_{{\mathcal D}_{X}}(
{\mathfrak M},
\Gamma_{M_{\alpha_{0}}}
(
{\mathcal B}_{M}
)
)
)
\\
&\to
\Gamma(U;
{\mathcal Hom}_{{\mathcal D}_{X}}(
{\mathfrak M},
{\mathcal B}_{M}
))
\to
\Gamma(U;
{\mathcal Hom}_{{\mathcal D}_{X}}(
{\mathfrak M},
\Gamma_{M\backslash M_{\alpha_{0}}}
(
{\mathcal B}_{M}
)
)),
\end{alignat*}
which is exact.
Therefore,
we have
\begin{eqnarray*}
&&\dim
\Gamma(U;
{\mathcal Hom}_{{\mathcal D}_{X}}(
{\mathfrak M},
{\mathcal B}_{M}
))\\
\leq &&
J_{\alpha_{0}}\cdot 
\operatorname{mult}_{T_{X_{\alpha_{0}}}^{*}X}^{{\mathcal D}_{X}}({\mathfrak M})
+
\dim\Gamma(U;
{\mathcal Hom}_{{\mathcal D}_{X}}(
{\mathfrak M},
\Gamma_{M\backslash M_{\alpha_{0}}}
(
{\mathcal B}_{M}
)
))
\end{eqnarray*}
by \eqref{dimJa}.
In the same way,
we have
\begin{eqnarray*}
&&
\dim\Gamma(U;
{\mathcal Hom}_{{\mathcal D}_{X}}(
{\mathfrak M},
\Gamma_{M\backslash M_{\alpha_{0}}}
(
{\mathcal B}_{M}
)
))
\\
\leq &&
J_{\alpha_{1}}\cdot 
\operatorname{mult}_{T_{X_{\alpha_{1}}}^{*}X}^{{\mathcal D}_{X}}({\mathfrak M})
+
\dim \Gamma(U;
{\mathcal Hom}_{{\mathcal D}_{X}}(
{\mathfrak M},
\Gamma_{M\backslash M_{\alpha_{0}}\cup M_{\alpha_{1}}}
(
{\mathcal B}_{M}
)
))
\end{eqnarray*}
for $\alpha_{1}\in A$
such that $X_{\alpha_{1}}$ is closed in $X\backslash X_{\alpha_{0}}$.
Repeating
this argument,
we have
the proposition.
\end{proof}

We note that $J_{\alpha}$ only depends on 
$U$ and
the stratification $X=\bigsqcup_{\alpha\in A} X_{\alpha}$.

\section{Upper bound of the dimensions of the spaces of group invariant hyperfunctions}
\label{proofSolfin}
In this section,
we give
an
upper bound of the dimensions of the spaces of group invariant hyperfunctions.
We want to use Proposition \ref{dim-rela-U}
for the proof of Theorem \ref{Solfin-analytic}.
Therefore,
we should show that
there exists a holonomic ${\mathcal D}_{X}$-module 
${\mathfrak M}_{\tau}$
and a regular stratification $X=\bigsqcup_{\alpha\in A}X_{\alpha}$
with respect to 
${\mathfrak M}_{\tau}$ such that
\begin{enumerate}[(1)]
    \item $
(\Gamma(U;{\mathcal B}_{M})\otimes \tau)^{{\mathfrak h}_{\mathbb C}}
\simeq
\Gamma(U;
{\mathcal Hom}_{{\mathcal D}_{X}}
({\mathfrak M}_{\tau},
{\mathcal B}_{M})
)
$ for any open subset $U\subset X$,
	\item the stratification $X=\bigsqcup_{\alpha\in A}X_{\alpha}$ does not depend on $\tau$,
	\item each stratum $X_{\alpha}$ is connected.
\end{enumerate}

First,
we construct a holonomic ${\mathcal D}_{X}$-module
${\mathfrak M}_{\tau}$
satisfying these conditions in Lemma \ref{constructM}.
\begin{Lemma}
\label{constructM}
	In the setting of 
	Theorem \ref{Solfin-analytic},
	there exists a holonomic ${\mathcal D}_{X}$-module
	${\mathfrak M}_{\tau}$
	satisfying the conditions (1), (2) and (3) above.
\end{Lemma}
\begin{proof}
	Let 
$U({\mathfrak h})$ be the universal enveloping algebra of ${\mathfrak h}_{\mathbb C}$.
The action of $H_{\mathbb C}$ on $X$
induces a Lie algebra homomorphism
$a\colon U({\mathfrak h})\to{\mathcal D}_{X}$.
By this homomorphism,
we regard
${\mathcal D}_{X}$ as a right $U({\mathfrak h})$-module.
We define a coherent
${\mathcal D}_{X}$-module ${\mathfrak M}_{\tau}$
by
(cf. Beilinson-Bernstein localization \cite{BB})
\begin{align}
{\mathfrak M}_{\tau}
:=
{\mathcal D}_{X}\otimes_{U({\mathfrak h})} \tau^{\vee}.
\label{DefM}
\end{align}
In other words,
\begin{align}
{\mathfrak M}_{\tau}
=
{\mathcal D}_{X}\otimes_{{\mathbb C}_{X}} \tau^{\vee}/
I_{\tau^{\vee}},
\end{align}
where
$I_{\tau^{\vee}}$ is a
${\mathcal D}_{X}$-submodule of
${\mathcal D}_{X}\otimes\tau^{\vee}$
defined by
\begin{eqnarray}
\label{Itau}
I_{\tau^{\vee}}
:=
\sum_{
H\in{\mathfrak h}_{\mathbb C}, v\in \tau^{\vee}}
{\mathcal D}_{X}
\cdot
(a(H)\otimes v-1\otimes \tau^{\vee}(H)v
).
\end{eqnarray}
Let 
$X=\bigsqcup_{\alpha\in A}X_{\alpha}$
be the $H_{\mathbb C}$-orbit decomposition of $X$
and
$X_{\alpha}=\bigsqcup_{k\in K_{\alpha}}X_{\alpha,k}$
the connected component decomposition of $X_{\alpha}$.
Then,
the finiteness of the number of connected components of $H_{\mathbb C}$ and $\#(H_{\mathbb C}\backslash X)<\infty$ imply that
${\mathfrak M}_{\tau}$ is a holonomic ${\mathcal D}_{X}$-module
and
$X=\bigsqcup_{\alpha\in A,k\in K_{\alpha}}X_{\alpha,k}$
is a regular stratification of $X$
with respect to
${\mathfrak M}_{\tau}$
(see \cite[Thm.\:5.1.12]{K})
such that each $X_{\alpha,k}$ is connected.
Thus
${\mathfrak M}_{\tau}$
	satisfies the conditions (2) and (3).
By
an isomorphism
$
{\mathcal B}_{M}\simeq
{\mathcal Hom}_{{\mathcal D}_{X}}
({\mathcal D}_{X},{\mathcal B}_{M})
$
and
the tensor-hom adjunction,
we have
\begin{align}
(\Gamma(M;{\mathcal B}_{M})\otimes\tau)^{\mathfrak h}
&\simeq
(\Gamma(M;{\mathcal Hom}_{{\mathcal D}_{X}}
({\mathcal D}_{X},{\mathcal B}_{M}))\otimes\tau)^{\mathfrak h}
\nonumber\\
&\simeq
(
{\rm Hom}_{{\mathcal D}_{X}}
({\mathcal D}_{X},{\mathcal B}_{M})\otimes\tau)^{\mathfrak h}
\nonumber\\
&\simeq
{\rm Hom}_{U({\mathfrak h})}(
\tau^{\vee},
{\rm Hom}_{{\mathcal D}_{X}}
({\mathcal D}_{X},{\mathcal B}_{M})
)\nonumber\\
&\simeq
{\rm Hom}_{{\mathcal D}_{X}}
({\mathcal D}_{X}\otimes_{U({\mathfrak h})}\tau^{\vee},
{\mathcal B}_{M})\nonumber\\
&\simeq
\Gamma(M;
{\mathcal Hom}_{{\mathcal D}_{X}}
({\mathfrak M}_{\tau},
{\mathcal B}_{M})
).
\nonumber
\end{align}
The similar argument shows that
${\mathfrak M}_{\tau}$ satisfies the condition (1).
\end{proof}
Reindexing,
we assume that $X=\bigsqcup_{\alpha \in A} X_{\alpha}$ is a regular stratification of $X$ with respect to ${\mathfrak M}_{\tau}$ such that each $X_{\alpha}$ is connected.
Then,
Proposition \ref{dim-rela-U} implies
\begin{eqnarray}
\dim
(\Gamma(U;{\mathcal B}_{M})\otimes \tau)^{{\mathfrak h}_{\mathbb C}}
\leq
\sum_{\alpha\in A}
J_{\alpha}
\cdot
{\rm mult}_{T^{*}_{X_{\alpha}}X}^{{\mathcal D}_{X}}
(
{\mathfrak M}_{\tau}
)
\label{<anal}
\end{eqnarray}
by Lemma \ref{constructM}.
Note that $J_{\alpha}$ dose not depend on $\tau$.
Thus,
we want to show that 
${\rm mult}_{T^{*}_{X_{\alpha}}X}^{{\mathcal D}_{X}}
(
{\mathfrak M}_{\tau}
)$ is uniformly bounded 
for the proof of Theorem \ref{Solfin-analytic}.

\begin{Lemma}
\label{multVtC}
Let
${\mathfrak M}_{\tau}$
be a holonomic ${\mathcal D}_{X}$-module defined in \eqref{DefM}.
Then,
for any 
$\alpha\in A$,
there exists 
$C_{\alpha}>0$,
which is independent of $\tau$,
satisfying
\begin{align*}
{\rm mult}_{T^{*}_{X_{\alpha}}X}^{{\mathcal D}_{X}}(
{\mathfrak M}_{\tau}
)
\leq
\dim\tau\cdot
C_{\alpha}.
\end{align*}
\end{Lemma}
For the proof of Lemma \ref{multVtC},
we need some preparation.
The order filtration ${\mathcal F}$ of
${\mathcal D}_{X}$
induces a filtration of a ${\mathcal D}_{X}$-module
${\mathcal D}_{X}\otimes \tau^{\vee}$,
which also induces good filtrations of
${\mathfrak M}_{\tau}$
and
$I_{\tau^{\vee}}$.
We write 
${\mathcal F}$
for
these filtration.
Namely, we put
\begin{eqnarray*}
{\mathcal F}_{j}({\mathcal D}_{X}\otimes \tau^{\vee})
&:=&
{\mathcal F}_{j}({\mathcal D}_{X})\otimes \tau^{\vee},
\\
{\mathcal F}_{j}({\mathfrak M}_{\tau})
&:=&
\left(
{\mathcal F}_{j}({\mathcal D}_{X}\otimes \tau^{\vee})+I_{\tau^{\vee}}
\right)
/
I_{\tau^{\vee}},\\
{\mathcal F}_{j}(I_{\tau^{\vee}})
&:=&
{\mathcal F}_{j}({\mathcal D}_{X}\otimes \tau^{\vee})\cap I_{\tau^{\vee}}.
\end{eqnarray*}
We note that there exists an isomorphism
\begin{eqnarray*}
gr_{\mathcal F}({\mathfrak M}_{\tau})
\simeq
gr_{\mathcal F}({\mathcal D}_{X}\otimes \tau^{\vee})
/
gr_{\mathcal F}(I_{\tau^{\vee}}).
\end{eqnarray*}
\begin{Lemma}
\label{surjection}
There exists an
${\mathcal O}_{T^{*}X}$-module
${\mathfrak N}$,
which is independent of $\tau$,
satisfying the following two conditions:
\begin{enumerate}
\item
$\dim 
{\rm supp}
\left(
{\mathfrak N}
\right)
\leq d_{X}$,

\item
there exists a surjective homomorphism $
F(
gr_{\mathcal F}(
{\mathfrak M}_{\tau}
))
\twoheadleftarrow
{\mathfrak N}
\otimes
\tau^{\vee}
$.
\end{enumerate}
\end{Lemma}
Postponing the proof of this lemma,
we prove Lemma 
\ref{multVtC}.
\begin{proof}[Proof of Lemma \ref{multVtC}]
By the definition of ${\rm mult}_{T^{*}_{X_{\alpha}}X}^{{\mathcal D}_{X}}(\cdot)$,
we have
\begin{align*}
{\rm mult}_{T^{*}_{X_{\alpha}}X}^{{\mathcal D}_{X}}({\mathfrak M}_{\tau})
=
{\rm mult}_{T^{*}_{X_{\alpha}}X}^{{\mathcal O}_{T^{*}X}}
(
F(
{\mathfrak M}_{\tau}
)
).
\end{align*}
Note that
${\rm mult}_{T^{*}_{X_{\alpha}}X}^{{\mathcal O}_{T^{*}X}}(\cdot)$
is additive
on short exact sequences of 
${\mathcal O}_{T^{*}X}$-modules with at most $d_{X}$-dimensional supports
\cite[Sect.\:1.5 in Appx.\:V]{Bjork}
and that
$
{\mathfrak N}
\otimes
\tau^{\vee}
\simeq
{\mathfrak N}^{\dim\tau}
$
as an ${\mathcal O}_{T^{*}X}$-module.
Then
Lemma \ref{surjection}
implies
\begin{align*}
{\rm mult}_{T^{*}_{X_{\alpha}}X}^{{\mathcal D}_{X}}({\mathfrak M}_{\tau})
\leq
{\rm mult}_{T^{*}_{X_{\alpha}}X}^{{\mathcal O}_{T^{*}X}}
\left(
{\mathfrak N}
\otimes
\tau^{\vee}
\right)
=\dim\tau\cdot
{\rm mult}_{T^{*}_{X_{\alpha}}X}^{{\mathcal O}_{T^{*}X}}
\left(
{\mathfrak N}
\right).
\end{align*}
This completes the proof.
\end{proof}
We shall prove Lemma \ref{surjection} from now on.
For any $j\in{\mathbb Z}_{\geq 0}$,
we write
$\sigma_{j}\colon {\mathcal F}_{j}({\mathcal D}_{X})
\to gr_{\mathcal F}({\mathcal D}_{X})$
for the composition of
the natural surjection
${\mathcal F}_{j}({\mathcal D}_{X})\to
{\mathcal F}_{j}({\mathcal D}_{X})/{\mathcal F}_{j-1}({\mathcal D}_{X})$
and
the injection
$
{\mathcal F}_{j}({\mathcal D}_{X})/{\mathcal F}_{j-1}({\mathcal D}_{X})
\hookrightarrow 
gr_{\mathcal F}({\mathcal D}_{X})$.
We use the same symbol
$\sigma_{j}$
for the homomorphism
$
{\mathcal F}_{j}({\mathcal D}_{X}\otimes \tau^{\vee})\to
gr_{\mathcal F}({\mathcal D}_{X}\otimes \tau^{\vee})$.
\begin{Lemma}
\label{sigma}
For any $H\in{\mathfrak h}_{\mathbb C}$
and
any $v\in \tau^{\vee}$,
we have
\begin{align*}
\sigma_{1}
(
a(H)\otimes v-1\otimes \tau^{\vee}(H)v
)
&=\sigma_{1}(a(H))\otimes v.
\end{align*}
\end{Lemma}
\begin{proof}
This is clear by
$1\otimes\tau^{\vee}(H)v\in {\mathcal F}_{0}({\mathcal D}_{X}\otimes \tau^{\vee})$.
\end{proof}
\begin{proof}[Proof of Lemma \ref{surjection}]
Recall that
$I_{\tau^{\vee}}$
is a 
${\mathcal D}_{X}$-submodule of
${\mathcal D}_{X}\otimes{\mathcal V}_{\tau^{\vee}}$
generated by
$\{a(H)\otimes v-1\otimes \tau^{\vee}(H)v\mid
H\in{\mathfrak h}_{\mathbb C}, v\in \tau^{\vee}\}$.
Then we have
\begin{align*}
gr_{\mathcal F}(I_{\tau^{\vee}})
\supset 
\sum_{H\in{\mathfrak h}_{\mathbb C}, v\in \tau^{\vee}}
gr_{\mathcal F}({\mathcal D}_{X})
\cdot
\sigma_{1}(a(H)\otimes v-1\otimes \tau^{\vee}(H)v)
\end{align*}
(cf. \cite[Chap.\:2]{gairon}).
Thus, Lemma \ref{sigma}
implies
\begin{align*}
gr_{\mathcal F}( I_{\tau^{\vee}})
\supset gr_{\mathcal F}({\mathcal D}_{X})
\cdot
(
\sigma_{1}(a({\mathfrak h}_{\mathbb C}))
\otimes
\tau^{\vee}
 ).
\end{align*}
This inclusion induces the following surjection:
\begin{align}
\frac{
gr_{\mathcal F}({\mathcal D}_{X}\otimes\tau^{\vee})
}{
gr_{\mathcal F}( I_{\tau^{\vee}})
}
&\twoheadleftarrow
\frac{
gr_{\mathcal F}({\mathcal D}_{X}\otimes\tau^{\vee})
}{
gr_{\mathcal F}({\mathcal D}_{X})
(
\sigma_{1}(a({\mathfrak h}_{\mathbb C}))
\otimes
\tau^{\vee}
 )
}
=
\left(
\frac{
gr_{\mathcal F}({\mathcal D}_{X})}
{
gr_{\mathcal F}(
{\mathcal D}_{X}
)
\sigma_{1}(a
({\mathfrak h}_{\mathbb C}))
}
\right)
\otimes
\tau^{\vee}.
\label{surjgr}
\end{align}
Note that the left-hand side is equal to
$gr_{\mathcal F}(
{\mathfrak M}_{\tau}
)$.
We define
an
${\mathcal O}_{T^{*}X}$-module
${\mathfrak N}$
by
${\mathfrak N}
:=
{\mathcal O}_{T^{*}X}/
{\mathcal O}_{T^{*}X}
\cdot
\sigma_{1}(a
({\mathfrak h}))$.
Applying
the right exact functor $F$ (see 
\eqref{Ffunctor})
to
the surjection
\eqref{surjgr},
we obtain the desired surjection
$$
F(
gr_{\mathcal F}(
{\mathfrak M}_{\tau}
))
\twoheadleftarrow
{\mathcal O}_{T^{*}X}
\otimes_{\pi^{-1}(gr({\mathcal D}_{X}))}
\pi^{-1}\left(
\frac{
gr_{\mathcal F}({\mathcal D}_{X})}
{
gr_{\mathcal F}(
{\mathcal D}_{X}
)
\sigma_{1}(a
({\mathfrak h}_{\mathbb C}))
}
\right)
\otimes
\tau^{\vee}
\simeq
{\mathfrak N}
\otimes
\tau^{\vee}.
$$

Next,
we prove
$\dim 
{\rm supp}
\left(
{\mathfrak N}
\right)
\leq d_{X}$.
Let
$X=\bigsqcup_{\alpha\in A}X_{\alpha}$
be the 
$H_{\mathbb C}$-orbit decomposition
of $X$.
Then we have
$\#(A)=\#(H_{\mathbb C}\backslash X)<\infty$
by the assumption.
Because
$
{\rm supp}
\left(
{\mathfrak N}
\right)
$
is contained in
$\bigcup_{\alpha\in A}T^{*}_{X_{\alpha}}X$
(cf.\:\cite[Thm.\:5.1.12]{K})
and
$\dim T^{*}_{X_{\alpha}}X=d_{X}$,
we have
$\dim 
{\rm supp}
\left(
{\mathfrak N}
\right)
\leq d_{X}$.
\end{proof}

\begin{proof}[Proof of Theorem \ref{Solfin-analytic}.]
By Lemma \ref{constructM},
we have \eqref{<anal} .
Note $\#(A)=\#(H_{\mathbb C}\backslash X)<\infty$ by the assumption.
Then, we have $J:=\sup_{\alpha\in A}J_{\alpha}<\infty$
by Lemma \ref{semiMalpha}.
Then, we have
\begin{eqnarray*}
	\dim
(\Gamma(U;{\mathcal B}_{M})\otimes \tau)^{{\mathfrak h}_{\mathbb C}}
&\leq &
J
\sum_{\alpha\in A}
{\rm mult}_{T^{*}_{X_{\alpha}}X}^{{\mathcal D}_{X}}
(
{\mathfrak M}_{\tau}
)\\
&\leq &
J
\sum_{\alpha\in A}
\dim\tau\cdot
C_{\alpha}
\end{eqnarray*}
by Lemma \ref{multVtC}.
Setting $C:=J\cdot \#(A)\cdot \sup_{\alpha\in A}C_{\alpha}$,
we have
Theorem \ref{Solfin-analytic}.
\end{proof}

\section{Proof of Theorem \ref{Qfin}}
\label{proofQfin}
In this section,
we prove Theorem \ref{Qfin}
by using
Theorem \ref{Solfin-analytic}.
For this end,
we quote the characterization of intertwining operators by
invariant 
distributions.
\begin{Fact}[{\cite[Prop.\:3.2]{KS}}]
\label{GPGP}
Let 
$G$
be a real Lie group.
Suppose that
$G'$ and $H$
are
closed subgroups of $G$
and
that
$H'$
is
a closed subgroup of $G'$.
Let $\tau$ and $\tau'$ be finite-dimensional representations of $H$ and $H'$, respectively.

\begin{enumerate}[(1)]
\item
There is a natural injective map:
\begin{eqnarray}
{\rm Hom}_{G'}\left(C^{\infty}(G/H,\tau),C^{\infty}(G'/H',\tau')\right)
\hookrightarrow
\left({\mathcal D}'(G/H,\tau^{\vee}\otimes{\mathbb C}_{2\rho})\otimes \tau'\right)^{H'}.
\label{surj}
\end{eqnarray}
Here
$\tau^{\vee}$
is the contragredient representation of $\tau$,
${\mathbb C}_{2\rho}$
is the one-dimensional representation of $H$ 
defined by
$h\mapsto \left|\: \det ({\rm Ad}(h)\colon {\mathfrak g}/{\mathfrak h}\to{\mathfrak g}/{\mathfrak h})\right|^{-1}$.
\item
If $H$ is cocompact in $G$ (e.g., a parabolic subgroup of $G$ or a uniform lattice), then (\ref{surj}) is a bijection.
\end{enumerate}
\end{Fact}

\begin{proof}[Proof of Theorem \ref{Qfin}]
By Fact \ref{GPGP},
we have
\begin{align}
{\rm Hom}_{G}
(C^{\infty}(G/Q,\eta),
C^{\infty}(G/H,\tau))
&\simeq
(
{\mathcal D}'
(G/Q,\eta^{\vee}\otimes{\mathbb C}_{2\rho})\otimes \tau)^{H}
\label{DQ}.
\end{align}
Let $Q$ act on 
${\mathcal D}'(G)$
from the right.
Regarding
$
{\mathcal D}'(G)\otimes
(\eta^{\vee}\otimes{\mathbb C}_{2\rho})
$
as a tensor representation of $Q$,
we have
\begin{align*}
{\mathcal D}'
(G/Q,\eta^{\vee}\otimes{\mathbb C}_{2\rho})
\simeq
\left({\mathcal D}'(G)\otimes
(\eta^{\vee}\otimes{\mathbb C}_{2\rho})\right)^{Q}.
\end{align*}
Moreover,
letting
$H$
(resp.
$Q$)
act
on
$\eta^{\vee}\otimes{\mathbb C}_{2\rho}$
(resp, $\tau$)
trivially,
we have
\begin{align*}
(\ref{DQ})
&\simeq
\left(
\left({\mathcal D}'(G)\otimes
(\eta^{\vee}\otimes{\mathbb C}_{2\rho})\right)^{Q}
\otimes \tau
\right)^{H}\\
&\simeq
\left({\mathcal D}'(G)\otimes
(\eta^{\vee}\otimes{\mathbb C}_{2\rho})
\otimes \tau
\right)^{H\times Q}\\
&\subset
\left({\mathcal B}_{G}(G)\otimes
(\eta^{\vee}\otimes{\mathbb C}_{2\rho})
\otimes \tau
\right)^{H\times Q}.
\end{align*}
The last inclusion 
follows from the fact that
the space of Schwartz distributions
can be imbedded in
the space of Sato's hyperfunctions.
In order to apply Theorem \ref{Solfin-analytic},
we shall construct a relatively compact semianalytic open subset $U$ of $G$.
Let 
$G=KAN$ be the Iwasawa decomposition.
This implies
$G$ is diffeomorphic to $K\times A\times N\simeq K\times {\mathbb R}_{>0}^{k}\times {\mathbb R}^{j}$
for some $k,j\in{\mathbb N}$,
where $K$ is compact.
Define a relatively compact semianalytic open subset $U$ of $G$ by $U:=K\times (1,2)^{k}\times (1,2)^{j}\subset K\times {\mathbb R}_{>0}^{k}\times {\mathbb R}^{j}\simeq G$.
Because
$QU=G$,
we have (cf. \cite[Thm.\:3.16]{KS})
\begin{align}
\left({\mathcal B}_{G}(G)\otimes
(\eta^{\vee}\otimes{\mathbb C}_{2\rho})
\otimes \tau
\right)^{H\times Q}
\subset
\left({\mathcal B}_{G}(U)\otimes
(\eta^{\vee}\otimes{\mathbb C}_{2\rho})
\otimes \tau
\right)^{{\mathfrak h}\oplus {\mathfrak q}}.
\label{BBB}
\end{align}
Therefore,
Theorem \ref{Qfin} follows from
Theorem \ref{Solfin-analytic}.
\end{proof}

\section{An alternative approach of (II)$\Rightarrow$(I) in Fact \ref{KOB}}
\label{APB}
\begin{proof}[Proof of (II)$\Rightarrow$(I) in Fact \ref{KOB}]
Let $\pi^{\vee}$ be the contragredient representation of $\pi\in\hat{G}_{\rm smooth}$
in the category of admissible smooth representations
with moderate growth.
By
Casselman's
subrepresentation theorem \cite{Casselman},
there exists an injection
$\pi^{\vee}\hookrightarrow C^{\infty}(G/P,\eta)$
for some $\eta\in \hat{P}_{\rm f}$.
Then 
$\pi$
is 
isomorphic to an irreducible quotient of
$
C^{\infty}(G/P,
\eta^{\vee}\otimes{\mathbb C}_{2\rho})$
because 
there 
exists 
a natural $G$-invariant pairing
$C^{\infty}(G/P,\eta)\times
C^{\infty}(G/P,
\eta^{\vee}\otimes{\mathbb C}_{2\rho})
\to
{\mathbb C}
$.
Moreover by Fact \ref{GPGP},
we have
\begin{alignat*}{1}
{\rm Hom}_{G}\left(C^{\infty}(G/P,\eta^{\vee}\otimes{\mathbb C}_{2\rho}),C^{\infty}(G/H,\tau)\right)
&\simeq
\left({\mathcal D}'(G/P,\eta)\otimes \tau\right)^{H}\\
&\subset
\left({\mathcal D}'(G/P_{0},\eta|_{P_{0}})\otimes \tau\right)^{H}\\
&\subset
\left({\mathcal D}'(G/P_{0},\eta|_{P_{0}})\otimes \tau\right)^{\mathfrak h},
\end{alignat*}
where
$P_{0}$ is 
the identity component of $P$.
We note that
any irreducible finite-dimensional representation of $P$
is the direct sum of
at most $\#(P/P_{0})$ irreducible representations of $P_{0}$.
Then,
it is sufficient to prove 
\begin{align}
\sup_{\tau\in\hat{H}_{\rm f}}
\sup_{\eta\in\hat{P}_{0,{\rm f}}}
\frac{1}{\dim \tau}
\dim
\left({\mathcal D}'(G/P_{0},\eta)\otimes \tau\right)^{\mathfrak h}<\infty.
\label{supD}
\end{align}
As in the original proof \cite{KO},
we use the Borel--Weil theorem to deal with the finite-dimensional representation $\eta$.
Let
$P=MAN$ be the Langlands decomposition of $P$.
Take a maximal abelian subspace ${\mathfrak t}$ of
${\mathfrak m}$ and
write ${\mathfrak t}_{\mathbb C}, {\mathfrak m}_{\mathbb C}$
for the complexifications of ${\mathfrak t}, {\mathfrak m}$, respectively.
We define some positivity on the root system of $({\mathfrak m}_{\mathbb C}, {\mathfrak t}_{\mathbb C})$
and write ${\mathfrak n}_{\mathfrak m}$ for the direct sum of positive root spaces of ${\mathfrak m}_{\mathbb C}$
relative to ${\mathfrak t}_{\mathbb C}$. 
Then, by the Borel--Weil theorem,
there exists $\lambda\in{{\mathfrak a}^{*}+\sqrt{-1}{\mathfrak t}}^{*}$ such that
$\eta\in\hat{P}_{0, {\rm f}}$ is isomorphic to $C^{\infty}(P_{0}/TAN, {\mathbb C}_{\lambda})^{{\mathfrak n}_{\mathfrak m}}$
as a $P_{0}$-representation.
Here 
${\mathbb C}_{\lambda}:=(\chi_{\lambda},{\mathbb C})$
is a one-dimensional representation of 
$TAN$
defined by
\begin{align*}
\chi_{\lambda}(e^{T+H}n):=e^{\lambda(T+H)}\quad\quad {\rm for}\:
T\in {\mathfrak t},
H\in{\mathfrak a},
n\in N
\end{align*}
and
${\mathfrak n}_{\mathfrak m}$ acts on $C^{\infty}(P_{0}/TAN, {\mathbb C}_{\lambda})$ from the right
by an isomorphism $P_{0}/TAN\simeq M_{0}/T\simeq M_{0,{\mathbb C}}/T_{\mathbb C}N_{\mathfrak m}$.
Therefore, we have
\begin{align}
\left({\mathcal D}'(G/P_{0},\eta)\otimes \tau\right)^{\mathfrak h}
&\simeq
\left({\mathcal D}'(G/P_{0},C^{\infty}(P_{0}/TAN, {\mathbb C}_{\lambda})^{{\mathfrak n}_{\mathfrak m}})\otimes \tau\right)^{\mathfrak h}\nonumber\\
&\subset
\left({\mathcal D}'(G/TAN, {\mathbb C}_{\lambda})^{{\mathfrak n}_{\mathfrak m}}\otimes \tau\right)^{\mathfrak h}.
\label{D}
\end{align}
The last inclusion is the composition of an isomorphism (see \cite{KP}, for example)
\begin{align}
{\mathcal D}'(G/P_{0},C^{\infty}(P_{0}/TAN, {\mathbb C}_{\lambda})^{{\mathfrak n}_{\mathfrak m}})
\simeq
({\mathcal D}'(G)\otimes C^{\infty}(P_{0}/TAN, {\mathbb C}_{\lambda})^{{\mathfrak n}_{\mathfrak m}})^{P_{0}}
\label{P0}
\end{align}
and an injection
\begin{align}
({\mathcal D}'(G)\otimes C^{\infty}(P_{0}/TAN, {\mathbb C}_{\lambda})^{{\mathfrak n}_{\mathfrak m}})^{P_{0}}&
\hookrightarrow
{\mathcal D}'(G/TAN, {\mathbb C}_{\lambda})^{{\mathfrak n}_{\mathfrak m}},\nonumber\\
\sum F(\cdot )\otimes f(\cdot)&
\mapsto 
\sum F(\cdot)\otimes f(e).
\label{inj}
\end{align}
Here, 
$P_{0}$ acts on ${\mathcal D}'(G)$ and 
$C^{\infty}(P_{0}/TAN, {\mathbb C}_{\lambda})^{{\mathfrak n}_{\mathfrak m}}$
from the right and the left, respectively, 
and
$e\in G$ denotes the identity element.
Let ${\mathfrak n}_{\mathfrak m}$
act on
${\mathbb C}_{\lambda}$
trivially.
Then, similarly to 
(\ref{P0}),
we have
\begin{align*}
{\mathcal D}'(G/TAN, {\mathbb C}_{\lambda})^{{\mathfrak n}_{\mathfrak m}}
\simeq
\left({\mathcal D}'(G/N)\otimes {\mathbb C}_{\lambda}\right)
^{{\mathfrak n}_{\mathfrak m}+{\mathfrak a}+{\mathfrak t}}.
\end{align*}
Let
${\mathfrak n}_{\mathfrak m}+{\mathfrak a}+{\mathfrak t}$
(resp.
${\mathfrak h}$)
act
on
$\tau$
(resp. ${\mathbb C}_{\lambda}$)
trivially.
Then we have
\begin{align}
(\ref{D})&\simeq
\left(\left({\mathcal D}'(G/N)\otimes {\mathbb C}_{\lambda}\right)
^{{\mathfrak n}_{\mathfrak m}+{\mathfrak a}+{\mathfrak t}}\otimes \tau\right)^{{\mathfrak h}}\nonumber\\
&\simeq
\left({\mathcal D}'(G/N)\otimes {\mathbb C}_{\lambda}\otimes\tau\right)
^{{\mathfrak h}\oplus({\mathfrak n}_{\mathfrak m}+{\mathfrak a}+{\mathfrak t})}
\nonumber
\\
&
\subset
\left({\mathcal B}_{G/N}(G/N)\otimes {\mathbb C}_{\lambda}\otimes\tau\right)
^{{\mathfrak h}\oplus({\mathfrak n}_{\mathfrak m}+{\mathfrak a}+{\mathfrak t})}.
\label{DB}
\end{align}
In order to apply Theorem \ref{Solfin-analytic},
we shall construct a relatively compact semianalytic open subset $U$ of $G$.
Let $K$ be a maximal compact subgroup $G$,
then
$G/N$ is diffeomorphic to $K\times A\simeq K\times {\mathbb R}_{> 0}^{k}$ for some $k\in{\mathbb N}$
by the Iwasawa decomposition.
We define a relatively compact open semianalytic set $U$ of $G/N$ by $U:=K\times (1,2)^{k}\subset K\times {\mathbb R}_{>0}^{k}\simeq G/N$.
Then we have
\begin{align}
\left({\mathcal B}_{G/N}(G/N)\otimes {\mathbb C}_{\lambda}\otimes\tau\right)
^{{\mathfrak h}\oplus({\mathfrak n}_{\mathfrak m}+{\mathfrak a}+{\mathfrak t})}
\simeq
\left({\mathcal B}_{G/N}(U)\otimes {\mathbb C}_{\lambda}\otimes\tau\right)
^{{\mathfrak h}\oplus({\mathfrak n}_{\mathfrak m}+{\mathfrak a}+{\mathfrak t})}
\label{BB}
\end{align}
in the same way as \eqref{BBB}.
Moreover,
the assumption $\#(H_{\mathbb C}\backslash G_{\mathbb C}/B)<\infty$
implies
$\#((H_{\mathbb C}\times A_{\mathbb C}T_{\mathbb C}N_{\mathfrak m})\backslash G_{\mathbb C}/N_{\mathbb C})<\infty$.
Therefore,
Theorem \ref{Solfin-analytic} implies that
there exists $C>0$,
which is independent on $\tau\in \hat{H}_{\rm f}$ and $\lambda\in{\mathfrak a}^{*}+\sqrt{-1}{\mathfrak t}^{*}$,
such that
\begin{align*}
\dim
\left({\mathcal B}_{G/N}(U)\otimes {\mathbb C}_{\lambda}\otimes\tau\right)
^{{\mathfrak h}\oplus({\mathfrak n}_{\mathfrak m}+{\mathfrak a}+{\mathfrak t})}
\leq 
C\cdot \dim\tau.
\end{align*}
This completes the proof.
\end{proof}

\appendix
\section{Appendix}
\label{Reg-mult}
In this section,
we prove the remaining assertion of Corollary \ref{CorRC}.
For this purpose,
we use the theory of regular holonomic ${\mathcal D}_{X}$-modules.
We recall some facts.
See \cite[Thms.\:5.5.21 and 22]{Bjork}
or \cite[Thm.\:6.1.3]{KK-Hol-III}.

\begin{Fact}
\label{fact-reg}
For any holonomic ${\mathcal D}_{X}$-module ${\mathfrak M}$,
there exists 
a regular holonomic ${\mathcal D}_{X}$-module ${\mathfrak M}_{\rm reg}$
such that 
\begin{enumerate}
	\item ${\mathcal D}_{X}^{\infty}\otimes_{{\mathcal D}_{X}}{\mathfrak M} \simeq {\mathcal D}_{X}^{\infty}\otimes_{{\mathcal D}_{X}}{\mathfrak M} _{\rm reg}$,
	\item ${\mathcal Hom}_{{\mathcal D}_{X}}({\mathfrak N},{\mathcal D}_{X}^{\infty}\otimes_{{\mathcal D}_{X}}{\mathfrak M})\simeq {\mathcal Hom}_{{\mathcal D}_{X}}({\mathfrak N},{\mathfrak M}_{\rm reg})$ for any reglar holonomic ${\mathcal D}_{X}$-module ${\mathfrak N}$. 
\end{enumerate}
\end{Fact}
Note that we have
${\rm mult}_{T^{*}_{X_{\alpha}}X}^{{\mathcal D}_{X}}({\mathfrak M})={\rm mult}_{T^{*}_{X_{\alpha}}X}^{{\mathcal D}_{X}}({\mathfrak M}_{\rm reg})$
by Remark \ref{mult-infty}.
Moreover,
Fact \ref{fact-reg} implies
\begin{eqnarray*}
	{\mathcal Hom}_{{\mathcal D}_{X}}({\mathfrak M}_{\rm reg},{\mathcal B}_{X_{\alpha}|X})
	&\simeq &
	{\mathcal Hom}_{{\mathcal D}_{X}}({\mathfrak M}_{\rm reg},
	{\mathcal B}_{X_{\alpha}|X}^{f}),
\end{eqnarray*}
where ${\mathcal B}_{X_{\alpha}|X}^{f}:=({\mathcal B}_{X_{\alpha}|X})_{\rm reg}$
because
${\mathcal B}_{X_{\alpha}|X}\simeq 
{\mathcal D}_{X}^{\infty}\otimes_{{\mathcal D}_{X}}{\mathcal B}_{X_{\alpha}|X}^{f}$ (see \cite[Thm.\:5.4.1]{KK-Hol-III}, for example).
Therefore,
it is sufficient to show
$$
\dim {\mathcal Hom}_{{\mathcal D}_{X}}({\mathfrak M}_{\rm reg},
	{\mathcal B}_{X_{\alpha}|X}^{f})_{x}
	\leq 
	{\rm mult}_{T^{*}_{X_{\alpha}}X}^{{\mathcal D}_{X}}({\mathfrak M}_{\rm reg}),
$$
for any $x\in X_{\alpha}$
in order to prove Corollary \ref{CorRC}.
\begin{proof}[Proof of Corollary \ref{CorRC}]
Note that the regular holonomic 
${\mathcal D}_{X}$-module
${\mathcal B}_{X_{\alpha}|X}^{f}$
satisfies
${\rm mult}_{T^{*}_{X_{\alpha}}X}^{{\mathcal D}_{X}}({\mathcal B}_{X_{\alpha}|X}^{f})=1$.
Let $x\in X_{\alpha}$ and
$
	f\in {\mathcal Hom}_{{\mathcal D}_{X}}({\mathfrak M}_{\rm reg},{\mathcal B}_{X_{\alpha}|X}^{f})_{x}
$
with $f\neq 0$.
Take an open neighborhood $U$ of $x$
such that 
$f$ is defined over $U$.
Consider an exact sequence on $U$
\begin{eqnarray}
	0 \to \operatorname{Ker}f
	\to 
	{\mathfrak M}_{\rm reg}
	\to
	{\mathcal B}_{X_{\alpha}|X}^{f}
	\to 0,
	\label{ex-KMB}
\end{eqnarray}
where the exactness at ${\mathcal B}_{X_{\alpha}|X}^{f}$ follows from ${\rm mult}_{T^{*}_{X_{\alpha}}X}^{{\mathcal D}_{X}}({\mathcal B}_{X_{\alpha}|X}^{f})=1$.
Then,
$\operatorname{Ker}f$ is regular holonomic
\cite[Prop.\:1.1.17]{KK-Hol-III}.
Moreover, additivity of ${\rm mult}_{T^{*}_{X_{\alpha}}X}^{{\mathcal D}_{X}}$
with respect to exact sequences of holonomic ${\mathcal D}_{X}$-modules \cite[Prop.\:2.6.15]{K}
implies 
$$
{\rm mult}_{T^{*}_{X_{\alpha}}X}^{{\mathcal D}_{X}}(\operatorname{Ker}f)=
{\rm mult}_{T^{*}_{X_{\alpha}}X}^{{\mathcal D}_{X}}({\mathfrak M}_{\rm reg})
-
1.
$$
Applying the left exact functor
${\mathcal Hom}_{{\mathcal D}_{X}}
(*,
{\mathcal B}_{X_{\alpha}|X}^{f})_{x}$ to \eqref{ex-KMB},
we have
\begin{eqnarray*}
	\dim {\mathcal Hom}_{{\mathcal D}_{X}}
({\mathfrak M}_{\rm reg},
{\mathcal B}_{X_{\alpha}|X}^{f})_{x}
\leq 
1 +
\dim {\mathcal Hom}_{{\mathcal D}_{X}}
(\operatorname{Ker}f,
{\mathcal B}_{X_{\alpha}|X}^{f})_{x} 
\end{eqnarray*}
by 
$\dim {\mathcal Hom}_{{\mathcal D}_{X}}
({\mathcal B}_{X_{\alpha}|X}^{f},
{\mathcal B}_{X_{\alpha}|X}^{f})_{x} =1$.
Repeating this argument by taking ${\mathfrak M}=\operatorname{Ker}f$,
we have the corollary
because 
we have
$\dim {\mathcal Hom}_{{\mathcal D}_{X}}
({\mathfrak M}_{\rm reg},
{\mathcal B}_{X_{\alpha}|X}^{f})_{x} =0$
if 
${\rm mult}_{T^{*}_{X_{\alpha}}X}^{{\mathcal D}_{X}}({\mathfrak M}_{\rm reg})=0$.
\end{proof}

\end{document}